\documentclass[parskip=never,abstract=true]{scrartcl}

\usepackage[utf8]{inputenc}
\usepackage{csquotes}

\usepackage[T1]{fontenc}

\usepackage{lmodern}
\usepackage{yfonts}
\usepackage{textcomp}
\usepackage{setspace}
\usepackage[a4paper,left=2.5cm,right=2cm,top=2.5cm,bottom=2.5cm]{geometry}

\usepackage{extarrows}
\usepackage{etoolbox}
\usepackage{xparse}

\usepackage{mathtools}
\usepackage{amsthm,thmtools}
\usepackage{dsfont}
\let\mathbb\mathds
\usepackage{amssymb}
\usepackage[scr]{rsfso} 
\usepackage{bm}
\usepackage{euscript}
\usepackage{amsbsy}
\DeclareMathAlphabet\mathbfcal{OMS}{cmsy}{b}{n}

\usepackage{microtype}
\usepackage{authblk}
\usepackage{tikz}
\usetikzlibrary{cd,quotes}
\usetikzlibrary{decorations.markings}
\usepackage{pgfplots}
\pgfplotsset{compat=1.13}
\usepackage{booktabs}
\usepackage{float}
\usepackage{environ}
\usepackage{xcolor,soul}
\usepackage{hyperref} 
\hypersetup{%
  colorlinks,%
  linkcolor={red!60!black},%
  citecolor={blue!50!black},%
  urlcolor={blue!80!black}%
}
\usepackage{enumitem,cleveref}
\usepackage{eso-pic}

\usepackage[normalem]{ulem}

\usepackage{tikz-cd}
\tikzset{Rightarrow/.style={double equal sign distance,>={Implies},->},
	triple/.style={-,preaction={draw,Rightarrow}},
	quadruple/.style={preaction={draw,Rightarrow,shorten >=0pt},shorten >=1pt,-,double,double
		distance=0.2pt}}

\def\on{\operatorname}

\def\CC{\mathbb{C}}

\def\C{\EuScript{C}}
\def\D{\EuScript{D}}

\def\BB{\mathbb{B}}

\def\DD{\mathbb{D}}
\def\bS{\textbf{S}}

\def\OO{\mathbb{O}}

\def\iCat{\EuScript{C}\!\on{at}}

\def\Hom{\on{Hom}}
\def\D{\EuScript{D}}

\def\id{\on{id}}
\def\Set{\on{Set}}

\SetEnumitemKey{axioms}{itemsep=0pt,align=left,itemindent=!,
                        listparindent=\parindent,parsep=\parskip,
                        before=,
                        label={\Roman*}}

\SetEnumitemKey{claims}{itemsep=0pt,align=left,itemindent=!,
                        listparindent=\parindent,parsep=\parskip,
                        before=,
                        label={\Roman*}}

\DeclarePairedDelimiterX\set[1]{\lbrace}{\rbrace}{#1}
\newlist{implications}{description}{1} 
\setlist[implications]{itemsep=0pt,leftmargin=\parindent}
\NewDocumentCommand\implication{o}
  {\IfValueTF{#1}
    {\auximplication#1\relax}
    {\item[\normalfont($\,\Rightarrow\,$)]}}
\NewDocumentCommand\auximplication{u-u\relax}
  {\item[\normalfont(#1)$\,\Rightarrow\,$(#2)]}

\makeatletter
\newcounter{diagram}[section]
\def\thediagram{\thesection.\arabic{diagram}}
\def\ftype@diagram{4}
\def\ext@diagram{diag}
\def\fnum@diagram{Diagram~\thediagram}
\def\fs@diagram{htbp!}
\ExplSyntaxOn
\NewDocumentEnvironment{diagram}{O{htbp!}m}
  {\@float{diagram}[#1]\centering}
  {
   
   \caption{}
   \label{#2}
   \end@float
  }

\newcounter{subdiagram}[diagram]
\def\thesubdiagram{\thediagram.\arabic{subdiagram}}
\NewEnviron{multidiagram}{\expandafter\domultidiagram\BODY\enddomultidiagram}
\NewDocumentCommand\domultidiagram{omu\enddomultidiagram}
 {
  \IfValueTF{#1}{\diagram[#1]}{\diagram}{}
   \refstepcounter{diagram}
   \centering
    \seq_clear:N \l_tmpb_seq
    \seq_set_split:Nnn \l_tmpa_seq { \next } { #3 }
    \seq_map_inline:Nn \l_tmpa_seq
     {
      \seq_put_right:Nn \l_tmpb_seq
       {
        \begin{tabular}[b]{@{}c@{}}
         ##1 \\[3ex]
         \refstepcounter{subdiagram}
         \label{#2\othercolon\the\value{subdiagram}}
         Diagram~\thesubdiagram 
        \end{tabular}
       }
     }
    \seq_use:Nn \l_tmpb_seq { \qquad }
   \let\label\@gobble
   \let\caption\@gobble
  \enddiagram
 }
\ExplSyntaxOff
\def\othercolon{:}
\makeatother

\declaretheoremstyle[bodyfont=\itshape,notefont=\bfseries]{abellanA}
\declaretheoremstyle[notefont=\bfseries]{abellanB}

\declaretheorem[style=abellanA,numberwithin=section,name={Theorem}]{theorem}

\declaretheorem[style=abellanA,numberlike=theorem,name={Lemma}]{lemma}
\declaretheorem[style=abellanB,numberlike=theorem,name={Definition}]{definition}
\declaretheorem[style=abellanB,numberlike=theorem,name={Remark}]{remark}
\declaretheorem[style=abellanB,numberlike=theorem,name={Construction}]{construction}
\declaretheorem[style=abellanA,numberlike=theorem,name={Proposition}]{proposition}

\declaretheorem[style=abellanB,numbered=no,name={Notation}]{notation}
\declaretheorem[style=abellanA,numberlike=theorem,name={Corollary}]{corollary}

\newtheorem*{thm*}{Theorem}
\newtheorem*{prop*}{Proposition}
\newtheorem*{cor*}{Corollary}

\docsvlist{theorem,lemma,definition,remark,proposition,example,corollary}

\let\leq\leqslant
\let\geq\geqslant
\let\epsilon\varepsilon
\let\isom\simeq

\newcommand*\tensor{\otimes}

\newcommand*\restr[3][\relax]{#2#1\rvert_{#3}^{}}
\newcommand*\mathblank{\mathord{-}}

\usepackage[bbgreekl]{mathbbol}

\let\emptyset\varnothing

\makeatletter
\newcommand{\fixed@sra}{$\vrule height 2\fontdimen22\textfont2 width 0pt\rightarrow$}
\newcommand{\shortarrowup}[1]{%
	\mathrel{\text{\rotatebox[origin=c]{65}{\fixed@sra}}}
}
\newcommand{\shortarrowdown}[1]{%
	\mathrel{\text{\rotatebox[origin=c]{250}{\fixed@sra}}}
}
\makeatother

\tikzcdset{
  arrows={},
  nat/.style={Rightarrow},
  nat1/.style={nat,shift left=.8ex},
  nat2/.style={nat,shift right=.8ex,swap},
  map/.style={},
  map1/.style={map,shift left=.7ex},
  map2/.style={map,shift right=.7ex,swap},
  mapA/.style={map,shift left=.5ex},
  mapB/.style={map,shift left=.5ex},
  unique/.style={line cap=round,densely dotted},
  incl/.style={hookrightarrow},
  mono/.style={rightarrowtail},
  epi/.style={twoheadrightarrow},
  iso/.style={map,sloped,"{\tikz[x=.6ex,y=.3ex]{\draw[-](0,0)sin(1,1)cos(2,0)sin(3,-1)cos(4,0);}}"},
  line/.style={dash},
  line1/.style={line,shift left=.3ex},
  line2/.style={line,shift right=.3ex,swap},
}

\DeclareMathOperator\Nsc{N^{sc}}

\makeatletter
\newcommand*\dirlim{\mathop{\mathpalette\varlim@{\rightarrowfill@\scriptscriptstyle}}\nmlimits@}
\newcommand*\prolim{\mathop{\mathpalette\varlim@{\leftarrowfill@\scriptscriptstyle}}\nmlimits@}
\makeatother

\newcommand{\nat}{\Rightarrow}

\tikzset{
  abellanarrows/.style={line cap=round,line join=round,line width=.4pt},
  abellanarrowlength/.store in=\abellanarrowlength,
}

\ExplSyntaxOn

\cs_generate_variant:Nn \tl_replace_all:Nnn { Nf }

\NewDocumentCommand \func { s O{} m }
 {
  \group_begin:
   \IfBooleanTF{#1}
    { \keys_set:nn { abellan / func } { aligned = true , #2 } }
    { \keys_set:nn { abellan / func } {#2} }
   \abellan_func:n {#3}
  \group_end:
 }
\NewDocumentCommand \arr { s o m }
 {
  \IfBooleanF{#1}
   { \bool_if:NT \l_abellan_aligned_bool { & } }
  \abellan_arr:n {#3}
 }
\NewDocumentCommand \addarr { o m m }
 {
  \keys_set:nn { abellan / func / addarrow } { name = {#2} , #3 }
  \tl_clear:N \l_abellan_arrname_tl
 }
\NewDocumentCommand \setupfunc { m } { \keys_set:nn { abellan / func } {#1} }

\bool_new:N \l_abellan_aligned_bool
\tl_new:N \l_abellan_func_tl
\tl_new:N \l_abellan_arrname_tl
\tl_new:N \l_abellan_arrmode_tl
\tl_new:N \g_abellan_func_substitutions_tl
\quark_new:N \q_abellan

\keys_define:nn { abellan / func }
 {
  aligned .bool_set:N = \l_abellan_aligned_bool ,
  aligned .initial:n = false ,
  mode .choices:nn = { base , tikz } { \tl_set_eq:NN \l_abellan_arrmode_tl \l_keys_choice_tl } ,
  mode .initial:n = base ,
  tikzarrowlength .code:n = \tikzset{abellanarrowlength={#1}} ,
  tikzarrowlength .initial:n = 2em ,
  tikzarrowstyle .code:n = \tikzset{abellanarrows/.append ~ style={#1}} ,
 }
\keys_define:nn { abellan / func / addarrow }
 {
  name .tl_set:N = \l_abellan_arrname_tl ,
  base .code:n =
    \abellan_addarrow:nnww { \l_abellan_arrname_tl } { base } #1 \q_abellan ,
  tikz .code:n =
    \abellan_addarrow:nnww { \l_abellan_arrname_tl } { tikz } #1 \q_abellan ,
  substitutions .code:n =
   \tl_map_inline:nn {#1}
    {
     \tl_gput_right:Nx \g_abellan_func_substitutions_tl
      {
       \exp_not:n { \tl_replace_all:Nnn \l_abellan_func_tl {##1} }
        { \arr{\l_abellan_arrname_tl} }
      }
    } ,
 }
\cs_new_protected:Npn \abellan_func:n #1
 {
  \resetdynamicto
  \tl_set:Nn \l_abellan_func_tl {#1}
  \tl_replace_all:Nfn \l_abellan_func_tl
   { \char_generate:nn { `: } { 12 } } { \colon }
  \bool_if:NTF \l_abellan_aligned_bool
   { \tl_replace_all:Nnn \l_abellan_func_tl { ; } { \\ } }
   { \tl_replace_all:Nnn \l_abellan_func_tl { ; } { ,\ } }
  \tl_use:N \g_abellan_func_substitutions_tl
  \bool_if:NT \l_abellan_aligned_bool { \! \begin {aligned} \scan_stop: }
  \tl_use:N \l_abellan_func_tl
  \bool_if:NT \l_abellan_aligned_bool { \\ \end {aligned} }
 }
\NewDocumentCommand \abellan_addarrow:nnww { m m O{} u\q_abellan }
 {
  \exp_args:Nc \NewDocumentCommand { abellan_arr_#1_#2:w } { #3 }
   {
    \use:c { abellan_arr_ \l_abellan_arrmode_tl :n } { #4 }
   }
 }
\cs_new_protected:Npn \abellan_arr:n #1
 {
  \use:c { abellan_arr_#1_\l_abellan_arrmode_tl :w }
 }
\cs_new_protected:Npn \abellan_arr_base:n #1
 {
  \mathrel{#1}
 }
\cs_new_protected:Npn \abellan_arr_tikz:n #1
 {
  \mathrel{\vcenter{\hbox{\tikz[abellanarrows]{#1}}}}
 }
\ExplSyntaxOff

\addarr{monomorphism}
 {
  substitutions = {>->}{ mono }{\rightarrowtail}{\mono} ,
  base = [o]{\IfValueTF{#1}{\rightarrowtail{#1}}{\mono}} ,
  tikz = [o]{\draw[>->,inner sep=0pt]
     (0,0)--({max(\abellanarrowlength,.5em\IfValueT{#1}{+width("$\scriptstyle#1$")})},0)
     \IfValueT{#1}{node[midway,above=.7ex]{\smash{$\scriptstyle#1$}}
     node[midway,below=.7ex]{}}
     ;} ,
  }
\addarr{mapsto}
 {
  substitutions = {\mapsto}{ mapsto } ,
  base = \mapsto ,
  tikz = {\draw[|->](0,0)--(\abellanarrowlength,0);} ,
 }
\addarr{rrarrows} 
 {
  substitutions = {\rightrightarrows}{ doublerr }{\doublerr} ,
  base = \rightrightarrows ,
  tikz = {\draw[>={To[width=.8ex,length=.4ex]},->](0,0)--(\abellanarrowlength,0);
          \draw[yshift=.8ex,>={To[width=.8ex,length=.4ex]},->](0,0)--(\abellanarrowlength,0);} ,
 }
\addarr{inclusion}
 {
  substitutions = {\hookrightarrow}{\incl}{ incl } ,
  base = [o]{\IfValueTF{#1}{\hookrightarrow{#1}}{\incl}} ,
  tikz = {\draw[{Hooks[right]}->](0,0)--(\abellanarrowlength,0);} ,
 }
\addarr{natural}
 {
 substitutions = {\Rightarrow}{=>}{ nat }{\nat} ,
  base = [o]{\IfValueTF{#1}{\xRightarrow{#1}}{\nat}} ,
  tikz = [o]{\draw[line cap=butt,double equal sign distance,-Implies,inner sep=0pt]
    (0,0)--({max(\abellanarrowlength,.3em\IfValueT{#1}{+width("$\scriptstyle#1$")})},0)
    \IfValueT{#1}{node[midway,above=.7ex]{\smash{$\scriptstyle#1$}}
    node[midway,below=.7ex]{}}
    ;} ,
 }
\addarr{epimorphism} 
 {
  substitutions = {\twoheadrightarrow}{->>}{ epi }{\epi},
   base = [o]{\IfValueTF{#1}{\xtwoheadrightarrow{#1}}{\epi}} ,
   tikz = [o]{\draw[->>,inner sep=0pt]
      (0,0)--({max(\abellanarrowlength,.5em\IfValueT{#1}{+width("$\scriptstyle#1$")})},0)
      \IfValueT{#1}{node[midway,above=.7ex]{\smash{$\scriptstyle#1$}}
      node[midway,below=.7ex]{}}
      ;} ,
   }
\addarr{to}
 {
  substitutions = {\to}{\rightarrow}{ to }{ funct }{ map }{->} ,
  base = [o]{\IfValueTF{#1}{\xrightarrow{#1}}{\to}} ,
  tikz = [o]{\draw[->,inner sep=0pt]
    (0,0)--({max(\abellanarrowlength,.5em\IfValueT{#1}{+width("$\scriptstyle#1$")})},0)
    \IfValueT{#1}{node[midway,above=.7ex]{\smash{$\scriptstyle#1$}}
    node[midway,below=.7ex]{}}
    ;} ,
 }
\newcommand*\resetdynamicto
  {\gdef\dynamicto{\arr*{to}\gdef\dynamicto{\arr*{mapsto}}}}
\resetdynamicto

\setupfunc{mode=tikz}

\ExplSyntaxOn
\NewDocumentCommand \printheader { m o m }
 {
  \par\noindent
  \begin{minipage}[t]{\textwidth}\noindent
  
  \begin{tabular}[t]{ll}
     & \keyval_parse:NNn \abellan_printname:n \abellan_printnamemail:nn { #3 } 
  \end{tabular}
  \vspace{.4cm}
  \end{minipage}
  \begin{center}\Large\bfseries
   #1 \IfValueT{#2}{\\[1ex] \large #2}
  \end{center}
  \vspace{.6cm}
 }
\tl_new:N \g_abellan_autores_tl
\tl_new:N \g_abellan_asignatura_tl
\cs_new_protected:Npn \abellan_printnamemail:nn #1 #2
 {
  #1 \\ & \quad\texttt{#2} \\ &
 }
\cs_new_protected:Npn \abellan_printname:n #1
 {
  #1 \\ &
 }
\ExplSyntaxOff

\makeatletter
\tikzcdset{
	open/.code     = {\tikzcdset{hook, circled};},
	open'/.code    = {\tikzcdset{hook', circled};},
	circled/.code  = {\tikzcdset{markwith = {\draw (0,0) circle (.375ex);}};},
	markwith/.code ={
		\pgfutil@ifundefined%
		{tikz@library@decorations.markings@loaded}%
		{\pgfutil@packageerror{tikz-cd}{You need to say %
				\string\usetikzlibrary{decorations.markings} to use arrows with markings}{}}{}%
		\pgfkeysalso{/tikz/postaction = {
				/tikz/decorate,
				/tikz/decoration={markings, mark = at position 0.5 with {#1}}}
		}
	},
}
\makeatother

\def\mbsSet{\on{Set}_{\Delta}^{\mathbf{mb}}}
\def\bS{\textbf{MB}}
\def\sS{\textbf{S}}

\def\scr{\EuScript}

\makeatletter
\newcommand{\myitem}[1]{%
	\item[#1]\protected@edef\@currentlabel{#1}%
}
\makeatother 

\title{2-Cartesian fibrations I: A model for $\infty$-bicategories fibred in $\infty$-bicategories}
\author{Fernando Abellán García \& Walker H. Stern}
\date{}
\begin{document}
	\maketitle
	\begin{abstract}
		In this paper, we provide a notion of $\infty$-bicategories fibred in $\infty$-bicategories which we call \emph{2-Cartesian fibrations}. Our definition is formulated using the language of marked biscaled simplicial sets: Those are scaled simplicial sets equipped with an additional collection of triangles containing the scaled 2-simplices, which we call \emph{lean triangles}, in addition to a collection of edges containing all degenerate 1-simplices.  We prove the existence of a left proper combinatorial simplicial model category whose fibrant objects are precisely the 2-Cartesian fibrations over a chosen scaled simplicial set $S$. Over the terminal scaled simplicial set, this provides a new model structure modeling $\infty$-bicategories, which we show is Quillen equivalent to Lurie's scaled simplicial set model. We conclude by providing a characterization of 2-Cartesian fibrations over an $\infty$-bicategory. This characterization then allows us to identify those 2-Cartesian fibrations arising as the coherent nerve of a fibration of $\on{Set}^+_{\Delta}$-enriched categories, thus showing that our definition recovers the preexisting notions of fibred 2-categories.
	\end{abstract}
	\tableofcontents
	\section{Introduction}
	Of Grothendieck's many insights, it may be the construction of a fibred category from a functor which most influenced the development of (higher) category theory. Right/left fibrations, (co)Cartesian fibrations, and the associated Grothendieck constructions have become integral parts of the $\infty$-categorical toolbox.  
	
	In our previous works \cite{AGS_ThmA}, \cite{AGS_TW}, and \cite{AG_cof}, we have made extensive and free use of this toolbox in our exploration of cofinality for $(\infty,2)$-categories. This paper can be seen as a necessary stepping-stone to the final phase of this exploration: providing a model structure for $(\infty,2)$-categories fibred in $(\infty,2)$-categories. The appropriate notion of (co)limits in $(\infty,2)$-categories has already been explored in \cite{GHL_LaxLim}, and the connection to the various forms of the Grothendieck construction are made exceptionally clear there. In the exposition in \cite{GHL_LaxLim}, however, the authors cleverly sidestep the need for additional technology to handle fibred $\infty$-bicategories, opting instead to work  with $\Set_\Delta^+$-enriched categories. 
	
	To generalize our cofinality criteria from \cite{AGS_ThmA} and \cite{AG_cof}, however, there is an unavoidable need for a such a theory of fibrations, as well as an associated Grothendieck construction. As stated in \cite{AG_cof} the cofinality criterion for ordinary $(\infty,1)$-categorical colimits can be understood in term of weighted colimits by proving an equivalence of weight functors with values in the category of spaces. One essential ingredient of the proof of this fact is the theory of Cartesian fibrations, i.e. functors with values in $\infty$-categories. We strongly believe that such proof will generalise in a straightforward manner once the necessary categorified theory of Cartesian fibrations is developed.  

  This paper is the first in a 2-part sequence which will provide this technology. We here define and develop a notion of \emph{2-Cartesian fibration}, using the language of simplicial sets equipped with a marking and two scalings. We then show the existence of a model structure whose fibrant objects are precisely these 2-Cartesian fibrations. In the second paper, we will develop the corresponding Grothendieck construction and establish the full $(\infty,2)$-categorical version of our results in \cite{AGS_ThmA} and \cite{AG_cof}. Needless to say, we expect the theory of 2-Cartesian fibrations to have a wide range of applications outside our cofinality framework in much the same way as Cartesian fibrations now occupy a central position in the study of $\infty$-categories.
	
	\subsection{Defining 2-Cartesian fibrations}
	  
    As one climbs up the ladder of categorification, the higher dimensionality manifests itself most obviously in the number of new variances that a functor can have. What seems like an innocent increase of complexity in the strict 2-categorical setting turns out to play a much more central role when working with $(\infty,2)$-categories. In particular, there should be \emph{four} sensible notions of $\infty$-bicategories fibred in $\infty$-bicategories. Loosely speaking, these correspond to a functor $f:\BB\to \CC$ having 
	\begin{enumerate}
		\item Cartesian lifts of 1-morphisms and coCartesian lifts of 2-morphisms;
    \item Cartesian lifts of 1-morphisms and Cartesian lifts of 2-morphisms;
    \item coCartesian lifts of 1-morphisms and Cartesian lifts of 2-morphisms;
		\item coCartesian lifts of 1-morphisms and coCartesian lifts of 2-morphisms.
	\end{enumerate}
  We will not explore all four of these notions here, instead focusing on case (1). Note that this will also immediately address case (4), since this dualization can be achieved by taking the opposite simplicial sets. We adopt the terminology employed by the authors in \cite{GHL_equiv} and denote the cases (1) and (4) as \emph{outer} 2-Cartesian (resp. 2-coCartesian) fibrations and similarly the cases (2) and (3) as \emph{inner} 2-Cartesian (resp. 2-coCartesian) fibrations. The reasons for our particular choice of variance (which could seem arbitrary to the reader) are related to the kind of cofinality we will discuss in our upcoming paper \cite{AGS_CartII}. For ease of reading, we call fibrations with our chosen variance simply \emph{2-Cartesian fibrations}, trusting that this terminology will be replaced in writings where it becomes unclear. We would like to emphasize that the techniques here presented paired with the discussion on inner fibrations in \cite{GHL_LaxLim} and \cite{LurieGoodwillie} can be adapted to produce the analogues of the results found in this paper.  
	
	There are several different kinds of clues in the literature which can help explain what 2-Cartesian fibrations should look like. Of particular interest is the work \cite{Buckley} of Buckley, which provides decategorified versions of these definitions  --- one in strict 2-categories, and one in bicategories. Unwinding Buckley's definitions (and dualizing appropriately), we find that they amount to the following conditions on a functor $F:B\to C$: 
	\begin{itemize}
		\item The induced functors
		\[
		F_{a,b}:B(a,b)\to C(a,b).
		\]
		are coCartesian fibrations. We call the coCartesian morphisms of these fibrations \emph{coCartesian 2-morphisms}. 
		\item The horizontal composite of coCartesian 2-morphisms is coCartesian. Here we are noting that the vertical composites of coCartesian morphisms are automatically coCartesian.
		\item For any 1-morphism $g: c\to F(b)$ in $C$, there is a Cartesian morphism $\tilde{g}:\tilde{c}\to b$ in $B$ lifting $g$.  
	\end{itemize}
	For the rest of this discussion we will choose to focus on coCartesian 2-morphisms. We hope that the reader will trust us in believing that there is not much novelty in defining Cartesian 1-morphisms since the definition generalises in an straightforward manner.
  
	One distinctly evident difficulty comes from the fact that a priori it seems that the definition of a coCartesian 2-morphism is dependent on the choice of a particular model for the mapping $\infty$-category. To justify our definition we will draw an analogy with the already well understood notion of an invertible 2-morphism, i.e. a thin 2-simplex. Let us suppose we are given a map of $\infty$-bicategories
  \[
    \func{p: \BB \to \CC}
  \]
  having the right lifting property against the class of scaled anodyne maps. We will call a 2-simplex $\sigma$ in $\BB$ \emph{left-degenerate} if $\sigma|_{\Delta^{\{0,1\}}}$ is degenerate. One readily verifies that a left-degenerate triangle $\sigma$ is thin if and only if $p(\sigma)$ is thin in $\CC$ and every lifting problem of the form
  \[
  \begin{tikzcd}[ampersand replacement=\&]
  \Lambda^n_0 \arrow[r,"f"] \arrow[d] \& \BB \arrow[d,"p"] \\
  \Delta^{n} \arrow[r] \& \CC
  \end{tikzcd}
  \] 
  such that $\restr{p}{\Delta^{\set{0,1,n}}}=\sigma$ admits a solution. To illuminate the previous claim, let us consider for every pair of objects $x,y \in \BB$ the mapping categories $\BB(x,y)$ defined in \cite[Section 2.3]{GHL_equiv}, whose simplices are maps $\Delta^{n+1}\to \BB$ sending the terminal vertex to $y$, and all other vertices to $x$. Then interpreting $\sigma$ as an edge in $\BB(x,y)$ we can see that the previous claim is essentially saying that an edge is an equivalence if and only if it is coCartesian and its image in $\CC(p(x),p(y))$ is an equivalence. Guided by this intuition we will say that a left-degenerate 2-simplex is \emph{coCartesian} if we can produce the dotted arrow below
  \[
  \begin{tikzcd}[ampersand replacement=\&]
  \Lambda^n_0 \arrow[r,"f"] \arrow[d] \& \BB \arrow[d,"p"] \\
  \Delta^{n} \arrow[r] \arrow[ur,dotted] \& \CC
  \end{tikzcd}
  \] 
 provided $\restr{f}{\Delta^{\set{0,1,n}}}=\sigma$. Our definition will be completed once it can be extended to an arbitrary 2-simplex. To this end we then note that, for any 2-simplex $\gamma$ in $\BB$, we can define a 3-simplex $\eta$
	\[
	\begin{tikzcd}
	& \gamma(2) & \\
	& & \\
	\gamma(0)\arrow[uur,"{\gamma_{02}}"]\arrow[dr,"\gamma_{01}"']\arrow[rr,equals] & & \gamma(0)\arrow[dl,"\gamma_{01}"]\arrow[uul,"g"'] \\
	& \gamma(1)\arrow[uuu,"\gamma_{12}",crossing over] & 
	\end{tikzcd}
	\]
	where $d^0(\eta)$ and $d^3(\eta)$ are scaled, $d^1(\eta)=\gamma$, and $d^2(\eta)$ is left-degenerate. This data exhibits and equivalence of composite 2-morphisms $d^3(\eta) \circ \gamma= d^0(\eta) \circ d^2(\eta)$. In particular, since $d^0(\eta), d^3(\eta)$ are thin we see that $\gamma$ and $d^2(\eta)$ represent the same 2-morphism in the mapping $\infty$-category of $\BB$. We will call $d^2(\eta)$ the \emph{left-degeneration} of $\gamma$. This allows us to make the following definition: A 2-simplex is coCartesian if and only if its left-degeneration is coCartesian.

	Once we have established our definition of coCartesian triangles, the definition of Cartesian edges is relatively straightforward. We simply require that lifting problems 
	\[
	\begin{tikzcd}
	\Lambda^n_n\arrow[r,"f"]\arrow[d] & \BB\arrow[d,"p"]\\
	\Delta^n\arrow[r] & \CC 
	\end{tikzcd}
	\]
	have solutions, provided that $f|_{\Delta^{\{0,n-1,n\}}}$ is a coCartesian triangle, and $f|_{\Delta^{\{n-1,n\}}}$ is $p$-Cartesian. 
	
	There are several additional technical conditions that come into play in our definition, but these two form the core idea. As we will discuss below, this intuitive approach becomes very close to the formal definition when the base of our fibration is a fibrant scaled simplicial set ($\infty$-bicategory). We refer the reader to the third section of the paper for a systematical approach of this definition.
	
	\subsection{Decorations and data}
	
	To handle the data of Cartesian morphisms in the theory of Cartesian fibrations of $(\infty,1)$-categories, Lurie introduces a decoration on simplicial sets. In \cite[Ch. 3]{HTT}, he defines a \emph{marked simplicial set} to consist of a simplicial set $X$ and a collection of edges $M_X\subset X_1$ which contains all degenerate edges. In this way, one can ``hardcode'' the Cartesian edges into a simplicial set. 
	
	We will follow a similar approach in our development of the 2-Cartesian model structure. Unfortunately, this entails rather a lot of decoration. A scaled simplicial set $(X,T_X)$ already comes equipped with a collection $T_X$ of `thin' triangles, which are taken to represent invertible 2-morphisms. To this we will add a \emph{second} collection of decorated 2-simplices, $C_X$, which we take to represent the \emph{coCartesian} 2-simplices. Since any invertible 2-morphism should, in particular, be coCartesian, we require $T_X\subset C_X$. We will call the elements of $C_X$ the \emph{lean} triangles.
	
	In addition to all of this, we then add a collection of marked 1-simplices. These we take to represent the Cartesian morphisms. All in all, the data we will have to consider to codify our intuition about 2-Cartesian fibrations consists of a simplicial set and \emph{three} decorations. We will call a simplicial set with these three decorations a \emph{marked biscaled simplicial set} (or \emph{mb simplicial set} for short) and denote the category of such by $\mbsSet$
	
	\subsection{Main results}
	
	As the title and preceding discussion already suggest, the main result of this paper is that there is a model structure on the category of marked biscaled simplicial sets over a scaled simplicial set $S$, whose fibrant objects are precisely the 2-Cartesian fibrations.
	
	\begin{thm*}[\ref{thm:modelstructure}]
		Let $S$ be a scaled simplicial set. Then there exists a left proper combinatorial simplicial model structure on $(\mbsSet )_{/S}$, which is characterized uniquely by the following properties:
		\begin{itemize}
			\item[C)] A morphism $f:X \to Y$ in $(\mbsSet )_{/S}$ is a cofibration if and only if $f$ induces a monomorphism betwee the underlying simplicial sets.
			\item[F)] An object $X \in (\mbsSet )_{/S}$ is fibrant if and only if $X$ is a 2-Cartesian fibration. 
		\end{itemize}
	\end{thm*}
	
	The proof of this theorem is, by necessity, quite technical. The main ingredients --- verification of a pushout-product axiom for an appropriate collection of anodyne maps and a characterization of weak equivalences by their properties on fibres --- dominate a large part of the paper. 
	
	From this model structure, we immediately obtain a model structure over the terminal scaled simplicial set --- equivalently a model structure on $\mbsSet$. As one would hope, this turns out to provide a new model for $\infty$-bicategories. This model is quite similar to the model of \cite{Verity}, which was shown to be equivalent to Lurie's model on $\Set_\Delta^{\mathbf{sc}}$ in \cite{GHL_equiv}. In our new model, however, there is a significant amount of redundant data: a second scaling which, for fibrant objects, agrees with the first scaling. 
	
	\begin{thm*}[\ref{thm:equivwithSC}]
		There is a a Quillen equivalence
		\[
		\begin{tikzcd}[ampersand replacement=\&]
		L:\on{Set}_{\Delta}^{\on{sc}} \arrow[r,shift left=0.8] \& \arrow[l,shift left=0.8] \mbsSet:U
		\end{tikzcd}
		\]
		between the model structure on mb simplicial sets over $\Delta^0$ and the model structure on scaled simplicial sets of \cite{LurieGoodwillie}.
	\end{thm*}
	
	Once these core results are established, we explore the cases where the base is a fibrant scaled simplicial set, i.e. an $\infty$-bicategory. In this setting, it is possible to give a much more intuitive characterization of the fibrant objects. 
	
	\begin{thm*}[\ref{thm:characterization2Fib}]
		Let $\BB$ be an $\infty$-bicategory and let $p:(X,M_X,T_X \subset C_X) \to \BB$ be an element of $(\mbsSet)_{/\BB }$. Then $p:X\to \BB$ is fibrant if and only if 
		\begin{enumerate}
			\item $p$ has the right lifting property against the generating scaled anodyne maps. 
			\item The collection $C_X$ lean of triangles in $X$ contains all left-degenerate coCartesian triangles.
			\item The collection $C_X$ is stable under composition along 1-morphisms.
			\item The collection $E_X$ consists of precisely the p-cartesian edges of $X$. 
			\item Every morphism in $\BB$ admits a $p$-Cartesian lift.
      \item Every 2-morphism in $\BB$ admits a coCartesian lift.
		\end{enumerate}
	\end{thm*}

	In addition to being a useful characterization of 2-Cartesian fibrations, this serves as a confirmation that our intuitive understanding cribbed from \cite{Buckley} was correct. Condition (1) is a formality in the strict 2-categorical case, and conditions (2)-(5) closely parallel those conditions which we extracted from \cite{Buckley}. We can, in fact, extract a further corollary from this characterization: 
	
	\begin{cor*}
		Let $F:\BB\to \CC$ be a 2-fibration in the sense of \cite{Buckley}, dualized to require coCartesian 2-morphisms rather than Cartesian 2-morphisms. Then the induced map 
		\[
		\Nsc(F): (\Nsc(\BB),M_{\BB},T_{\BB}\subset C_{\BB})\to \Nsc(\CC)
		\]
		is a 2-Cartesian fibration, where $M_{\BB}$ is the set of Cartesian edges, and $C_{\BB}$ is the set of coCartesian triangles.
	\end{cor*}
	
	\subsection{Structure of this paper}
	
	In \autoref{sec:prelim}, we briefly introduce some notational conventions and recapitulate the generating scaled anodyne morphisms of \cite{LurieGoodwillie}. The next section, \autoref{sec:Model} is entirely given over to the proof that the desired 2-Cartesian model structure exists. This proof has two lengthy technical components: in \autoref{subsec:Anodyne}, we define our chosen class of anodyne morphisms --- the $\bS$-anodyne morphisms --- and prove that they satisfy a pushout-product axiom; in \autoref{subsec:MS}, we provide a fibrewise characterization of the putative equivalences, and then deduce the existence of the model structure. The final section, \autoref{sec:fibbase}, is devoted to providing a clear characterization of the fibrant objects over an $\infty$-bicategory. In particular, the characterizations provided in this section make clear the connection between our 2-Cartesian fibrations, the $\Set_\Delta^+$-enriched fibrations of \cite{GHL_LaxLim}, and the 2-fibrations of \cite{Buckley}.
	
	\subsection*{Acknowledgements}
	F.A.G. would like to acknowledge the support of the VolkswagenStiftung through the Lichtenberg Professorship Programme while he conducted this research. W.H.S wishes to acknowledge the support of the NSF Research Training Group at the
	University of Virginia (grant number DMS-1839968) during the preparation of this work.

	\section{Preliminaries}\label{sec:prelim}
	
	Recapitulating even the basics of the theory of quasi-categories, $\infty$-bicategories, and the various types of fibrations between them would take more space than the rest of the paper. Consequently, we here confine ourselves to fixing some notational conventions, and establishing definitions for later reference.
	
	We will denote the simplex category by $\Delta$, the category of simplicial sets by $\Set_\Delta$, the category of marked simplicial sets by $\Set_\Delta^+$, and the category of scaled simplicial sets by $\Set_\Delta^{\mathbf{sc}}$. Where possible, we will follow the notational conventions established in \cite{HTT} and expanded in \cite{LurieGoodwillie} and \cite{HA}. In referring to the works of Gagna, Harpaz, and Lanari \cite{GHL Gray,GHL_LaxLim,GHL_equiv}, we will endeavor to either follow their notation, or explain where our conventions differ.  
	
	\begin{definition}\label{def:scanodyne}
		The set of \emph{generating scaled anodyne maps} \(\sS\) is the set of maps of scaled simplicial sets consisting of:
		\begin{enumerate}
			\item[(i)]\label{item:anodyne-inner} the inner horns inclusions
			\[
			\bigl(\Lambda^n_i,\{\Delta^{\{i-1,i,i+1\}}\}\bigr)\rightarrow \bigl(\Delta^n,\{\Delta^{\{i-1,i,i+1\}}\}\bigr)
			\quad , \quad n \geq 2 \quad , \quad 0 < i < n ;
			\]
			\item[(ii)]\label{i:saturation} the map 
			\[
			(\Delta^4,T)\rightarrow (\Delta^4,T\cup \{\Delta^{\{0,3,4\}}, \ \Delta^{\{0,1,4\}}\}),
			\]
			where we define
			\[
			T\overset{\text{def}}{=}\{\Delta^{\{0,2,4\}}, \ \Delta^{\{ 1,2,3\}}, \ \Delta^{\{0,1,3\}}, \ \Delta^{\{1,3,4\}}, \ \Delta^{\{0,1,2\}}\};
			\]
			\item[(iii)]\label{item:anodyne_outer} the set of maps
			\[
			\Bigl(\Lambda^n_0\coprod_{\Delta^{\{0,1\}}}\Delta^0,\{\Delta^{\{0,1,n\}}\}\Bigr)\rightarrow \Bigl(\Delta^n\coprod_{\Delta^{\{0,1\}}}\Delta^0,\{\Delta^{\{0,1,n\}}\}\Bigr)
			\quad , \quad n\geq 3.
			\]
		\end{enumerate}
		A general map of scaled simplicial set is said to be \emph{scaled anodyne} if it belongs to the weakly saturated closure of \(\sS\).
	\end{definition}
	
	\begin{definition}
		We say that a map of scaled simplicial sets $p:X \to S$ is a weak \sS-fibration if it has the right lifting property with respect to the class of scaled anodyne maps.
	\end{definition}

	In general, we will denote fibrant objects in $\Set_\Delta^{\mathbf{sc}}$ using blackboard bold characters, e.g. $\BB$. We will use undecorated roman majescules, e.g. $X$, to denote objects of any category, adding explicit decorations as necessary for clarity.

  \begin{definition}\label{defn:mappingspace}
  Consider the cosimplicial object 
  \[
  \func*{
    Q:\Delta \to \Set_\Delta^{\on{sc}};
    {[n]} \mapsto \Delta^0 \coprod_{\Delta^n} (\Delta^n\star \Delta^0),
  } 
  \]
  equipped with the minimal scaling. Given an $\infty$-bicategory $X\in \Set_\Delta^{\on{sc}}$, for any $a,b\in X$, we define a simplicial set $X(a,b)$ whose $n$-simplices are maps $Q_n\to X$ which send the first vertex to $a$ and the second to $b$. It was shown in \cite{GHL_equiv} that $X(a,b)$ is a model for the mapping $\infty$-category from $a$ to $b$ in $X$.   
\end{definition}
	
	\section{The model structure}\label{sec:Model}
	\subsection{Marked biscaled simplicial sets and \bS-anodyne morphisms.}\label{subsec:Anodyne}
	\begin{definition}
		A \emph{marked biscaled} simplicial set (mb simplicial set) is given by the following data
		\begin{itemize}
			\item A simplicial set $X$.
			\item A collection of edges  $E_X \in X_1$ containing all degenerate edges.
			\item A collection of triangles $T_X \in X_2$ containing all degenerate triangles. We will refer to the elements of this collection as \emph{thin triangles}.
			\item A collection of triangles $C_X \in X_2$ such that $T_X \subseteq C_X$. We will refer to the elements of this collection as \emph{lean triangles}.
		\end{itemize}
		We will denote such objects as triples $(X,E_X, T_X \subseteq C_X)$. A map $(X,E_X, T_X \subseteq C_X) \to (Y,E_Y,T_Y \subseteq C_Y)$ is given by a map of simplicial sets $f:X \to Y$ compatible with the collections of edges and triangles above. We denote by $\mbsSet$ the category of mb simplicial sets.
	\end{definition}
	
	\begin{notation}
		Let $(X,E_X, T_X \subseteq C_X)$ be a mb simplicial set. Suppose that the collection $E_X$ consist only of degenerate edges. Then we fix the notation $(X,E_X, T_X \subseteq C_X)=(X,\flat,T_X \subseteq E_X)$ and do similarly for the collection $T_X$. If $C_X$ consists only of degenerate triangles we fix the notation $(X,E_X, T_X \subseteq C_X)=(X,E_X, \flat)$. In an analogous fashion we wil use the symbol “$\sharp$“ to denote a collection containing all edges (resp. all triangles). Finally suppose that $T_X=C_X$ then we will employ the notation $(X,E_X,T_X)$.
	\end{notation}
	
	\begin{remark}
		We will often abuse notation when defining the collections $E_X$ (resp. $T_X$, resp. $C_X$) and just specified its non-degenerate edges (resp. triangles).
	\end{remark}
	
	\begin{remark}\label{rem:forgetfuladj}
		Observe that we have a functor, $\func{L: \on{Set}_{\Delta}^{\on{sc}} \to \mbsSet}$ sending an scaled simplicial set $(X,T_X)$ to  $(X,\flat, T_X)$ which is left adjoint to the forgetful functor $U$ sending $(X,E_X,T_X \subseteq C_X)$ to $(X,T_X)$. 
	\end{remark}
	
	\begin{definition}\label{def:mbsanodyne}
		The set of \emph{generating mb anodyne maps} \(\bS\) is the set of maps of mb simplicial sets consisting of:
		\begin{enumerate}
			\myitem{(A1)}\label{mb:innerhorn} The inner horn inclusions 
			\[
			\bigl(\Lambda^n_i,\flat,\{\Delta^{\{i-1,i,i+1\}}\}\bigr)\rightarrow \bigl(\Delta^n,\flat,\{\Delta^{\{i-1,i,i+1\}}\}\bigr)
			\quad , \quad n \geq 2 \quad , \quad 0 < i < n ;
			\]
			\myitem{(A2)}\label{mb:wonky4} The map 
			\[
			(\Delta^4,\flat,T)\rightarrow (\Delta^4,\flat,T\cup \{\Delta^{\{0,3,4\}}, \ \Delta^{\{0,1,4\}}\}),
			\]
			where we define
			\[
			T\overset{\text{def}}{=}\{\Delta^{\{0,2,4\}}, \ \Delta^{\{ 1,2,3\}}, \ \Delta^{\{0,1,3\}}, \ \Delta^{\{1,3,4\}}, \ \Delta^{\{0,1,2\}}\};
			\]
			\myitem{(A3)}\label{mb:leftdeglefthorn} The set of maps
			\[
			\Bigl(\Lambda^n_0\coprod_{\Delta^{\{0,1\}}}\Delta^0,\flat,\flat \subset\{\Delta^{\{0,1,n\}}\}\Bigr)\rightarrow \Bigl(\Delta^n\coprod_{\Delta^{\{0,1\}}}\Delta^0,\flat,\flat \subset\{\Delta^{\{0,1,n\}}\}\Bigr)
			\quad , \quad n\geq 2.
			\]
			These maps force left-degenerate lean-scaled triangles to represent coCartesian edges of the mapping category.
			\myitem{(A4)}\label{mb:2Cartesianmorphs} The set of maps
			\[
			\Bigl(\Lambda^n_n,\{\Delta^{\{n-1,n\}}\},\flat \subset \{ \Delta^{\{0,n-1,n\}} \}\Bigr) \to \Bigl(\Delta^n,\{\Delta^{\{n-1,n\}}\},\flat \subset \{ \Delta^{\{0,n-1,n\}} \}\Bigr) \quad , \quad n \geq 2.
			\]
			This forces the marked morphisms to be $p$-Cartesian with respect to the given thin and lean triangles. 
			\myitem{(A5)}\label{mb:2CartliftsExist} The inclusion of the terminal vertex
			\[
			\Bigl(\Delta^{0},\sharp,\sharp \Bigr) \rightarrow \Bigl(\Delta^1,\sharp,\sharp \Bigr).
			\]
			This requires $p$-Cartesian lifts of morphisms in the base to exist.
			\myitem{(S1)}\label{mb:composeacrossthin} The map
			\[
			\Bigl(\Delta^2,\{\Delta^{\{0,1\}}, \Delta^{\{1,2\}}\},\sharp \Bigr) \rightarrow \Bigl(\Delta^2,\sharp,\sharp \Bigr),
			\]
			requiring that $p$-Cartesian morphisms compose across thin triangles.
			\myitem{(S2)}\label{mb:coCartoverThin} The map
			\[
			\Bigl(\Delta^2,\flat,\flat \subset \sharp \Bigr) \rightarrow \Bigl( \Delta^2,\flat,\sharp\Bigr),
			\]
			which requires that lean triangles over thin triangles are, themselves, thin.
			\myitem{(S3)}\label{mb:innersaturation} The map
			\[
			\Bigl(\Delta^3,\flat,\{\Delta^{\{i-1,i,i+1\}}\}\subset U_i\Bigr) \rightarrow \Bigl(\Delta^3,\flat, \{\Delta^{\{i-1,i,i+1\}}\}\subset \sharp \Bigr) \quad, \quad 0<i<3
			\]
			where $U_i$ is the collection of all triangles except $i$-th face. This and the next two generators serve to establish composability and limited 2-out-of-3 properties for lean triangles.
			\myitem{(S4)}\label{mb:dualcocart2of3} The map
			\[
			\Bigl(\Delta^3 \coprod_{\Delta^{\{0,1\}}}\Delta^0,\flat,\flat \subset U_0\Bigr) \rightarrow \Bigl(\Delta^3 \coprod_{\Delta^{\{0,1\}}}\Delta^0,\flat, \flat \subset \sharp \Bigr) 
			\]
			where $U_0$ is the collection of all triangles except the $0$-th face.
			\myitem{(S5)}\label{mb:coCart2of3} The map
			\[
			\Bigl(\Delta^3,\{\Delta^{\{2,3\}}\},\flat \subset U_3\Bigr) \rightarrow \Bigl(\Delta^3,\{\Delta^{\{2,3\}}\}, \flat \subset \sharp \Bigr) 
			\]
			where $U_3$ is the collections of all triangles except the $3$-rd face.
			\myitem{(E)}\label{mb:equivalences} For every Kan complex $K$, the map
			\[
			\Bigl( K,\flat,\sharp  \Bigr) \rightarrow \Bigl(K,\sharp, \sharp\Bigr).
			\]
			Which requires that every equivalence is a marked morphism.
		\end{enumerate}
		A map of mb simplicial sets is said to be \bS-anodyne if it belongs to the weakly saturated closure of \bS.
	\end{definition}
	
	\begin{definition}
		Let $f:(X,E_X,T_X \subseteq C_X) \to (Y,E_Y,T_Y \subseteq C_Y)$ be a map of mb simplicial sets. We say that $f$ is a \bS-fibration if it has the right lifting property against the class of \bS-anodyne morphisms.
	\end{definition}
	
	\begin{lemma}\label{lem:overthepoint}
		Let $f:(X,E_X,T_X \subseteq C_X) \to (Y,E_Y,T_Y \subseteq C_Y)$ be a \bS-fibration and denote by $X_y$ the fibre of $f$ over $y \in Y$. Then $X_y$ is an $\infty$-bicategory with precisely the equivalences marked.
	\end{lemma}
	\begin{proof}
		Observe that it follows from  \ref{mb:coCartoverThin} that the thin triangles and the lean triangles of $X_y$ must coincide. Since $X_y$ has the right lifting property against maps \ref{mb:innerhorn}-\ref{mb:leftdeglefthorn} it follows that $X_y$ is an $\infty$-bicategory. It follows from \ref{mb:equivalences} that all equivalences must be marked. It is a trivial computation to verify that every marked edge in $X_y$ is an equivalence.
	\end{proof}
	
	\begin{lemma}\label{lem:2Cart2of3}
		The morphism 
		\[
		\func{
			\theta:\Bigr(\Delta^2, \{\Delta^{\{1,2\}},\Delta^{\{0,2\}}\},\sharp\Bigl) \to \Bigr(\Delta^2,\sharp,\sharp\Bigl)
		}
		\]
		is \bS-anodyne.
	\end{lemma}
	
	\begin{proof}
		We first note that, given a \bS-fibration $f:(X,E_X,T_X \subseteq C_X) \to (S,\sharp,T_S )$, we can find a lift of $\theta$ as follows. Suppose we have a lifting problem 
		\[
		\begin{tikzcd}
		\Bigr(\Delta^2, \{\Delta^{\{1,2\}},\Delta^{\{0,2\}}\},\sharp\Bigl)\arrow[r,"\sigma"]\arrow[d,"\theta"'] & X\arrow[d,"f"] \\
		\Bigr(\Delta^2,\sharp,\sharp\Bigl)\arrow[r] & S
		\end{tikzcd}
		\]
		Where the top arrow corresponds to the thin 2-simplex 
		\[
	\sigma:\qquad 	\begin{tikzcd}[column sep=1em]
		 & b\arrow[dr,circled,"u"] & \\
		 a\arrow[ur,"v"]\arrow[rr,circled,"w"'] & & c
		\end{tikzcd}
		\]
		Since $f:X\to S$ is a \bS-fibration, we can choose a marked lift $\hat{v}:\hat{a}\to b$ of $f(v)$. Using a lift of type \ref{mb:innerhorn} to compose $u$ and $\hat{v}$ and a lift of type \ref{mb:2Cartesianmorphs} to obtain a morphism from $a$ to $\hat{a}$, we can obtain a $\Lambda^3_2$-horn, all of whose sides are thin-scaled. We can fill this to a maximally thin-scaled 3-simplex using a pushout of type \ref{mb:innerhorn} and a pushout of type \ref{mb:wonky4}. This three-simplex has the form 
		\[
		\begin{tikzcd}
		 & c & \\
		 & & \\
		 a\arrow[uur,circled,"w"]\arrow[dr,"v"']\arrow[rr,"p"',pos=0.3] & & \hat{a}\arrow[dl,circled,"\hat{v}"]\arrow[uul,"q"'] \\
		  & b\arrow[uuu,circled,"u"] & 
		\end{tikzcd}
		\] 
		Since every triangle is scaled, we can apply lifts of type \ref{mb:composeacrossthin} to show that $q$ is marked. This implies that $p$ is and equivalence in the fibre over $f(a)$, and so $p$ is marked. Thus, a lift of type \ref{mb:composeacrossthin} shows that $v$ is marked as desired.
		
		The fact that $\theta$ is \bS-anodyne then follows from the small object argument, which allows us to display $\theta$ as a retract of the \bS-anodyne replacement of $\theta$. 
	\end{proof}
	
	\begin{definition}
		We say that a map $f:(X,E_X,T_X \subseteq C_X) \to (Y,E_Y,T_Y \subseteq C_Y)$ in $s$ is a cofibration if its underlying map of simplicial sets is a monomorphism.
	\end{definition}
	
	\begin{remark}
		The generators of the class of cofibrations are given by
		\begin{itemize}
			\myitem{(C1)}\label{cof:bndry} $\Bigr(\partial \Delta^n,\flat,\flat\Bigl) \rightarrow \Bigr( \Delta^n,\flat,\flat\Bigl)$.
			\myitem{(C2)}\label{cof:marked1} $\Bigr(\Delta^1,\flat,\flat \Bigl) \rightarrow \Bigr(\Delta^1,\sharp,\flat \Bigl) $.
			\myitem{(C3)}\label{cof:coCart} $\Bigr(\Delta^2,\flat,\flat\Bigl) \rightarrow \Bigr(\Delta^2,\flat,\flat \subset \sharp)$.
			\myitem{(C4)}\label{cof:thin} $\Bigl(\Delta^2,\flat,\flat \subset \sharp \Bigr) \rightarrow \Bigl( \Delta^2,\flat,\sharp\Bigr)$.
		\end{itemize}
		Note that \ref{cof:thin} and \ref{mb:coCartoverThin} are the same morphism.
	\end{remark}
	
	\begin{proposition}\label{prop:PP}
		Let $f:(X,E_X,T_X \subseteq C_X) \to (Y,E_Y,T_Y \subseteq C_Y)$ be a \bS-anodyne morphism and let $g:(A,E_A,T_A \subseteq C_A) \to (B,E_B,T_B \subseteq C_B)$ be a cofibration. Then the pushout-product
		\[
		\func{f \wedge g: X \times B \coprod\limits_{X \times A}Y \times A \to Y \times B }
		\]
		is \bS-anodyne.\footnote{Note that this proposition is about the pushout-product of marked biscaled simplicial sets. For readability, we have omitted the marking and biscaling from the notation in the conclusion.}
	\end{proposition}
	
	Before embarking on our proof of the pushout-product, we will tackle one particularly recalcitrant case by itself. As it so happens, a case nearly precisely dual to this one also occurs in checking the pushout-product. To save paper (and the reader's eyesight), we will only provide the proof of one of these cases, trusting that it will be apparent how to dualize the argument.   
	
	We we first prove two quick lemmata, which will somewhat ease the coming proof.
	
	\begin{construction}\label{cons}
		We construct a marking and biscaling on $\Delta^m$. For an non-consecutive list of vertices $\vec{i}=\{i_1,\ldots i_{k+1}\}$ which does not contain $0$ of $m$, denote by $\Lambda^{m}_{\vec{i}}$ the simplicial subset whose simplices are those whose corresponding subsets $J\subset [m]$ such that $J$ skips a vertex $j \in [m]$ such that $j \notin \vec{i}$. Define a biscaling $T_{\vec{i}}$ on $\Delta^{m}$ by requiring that $\Delta^{\{i-1,i,i+1\}}$ is thin for every $i\in I$.  
	\end{construction}
	
	\begin{lemma}\label{lem:indI}
		The inclusion 
		\[
		\func{
			( \Lambda^{m}_{\vec{i}},\flat, T_{\vec{i}}) \to  (\Delta^{m},\flat,T_{\vec{i}})
		}
		\]
		is \bS-anodyne for $m\geq 2$.
	\end{lemma}
	
	\begin{proof}
		We proceed by induction on the length of $\vec{i}$. When $\on{length}(\vec{i})=1$, this is simply a morphism of type \ref{mb:innerhorn}. 
		
		Now suppose that this holds for $\on{length}(\vec{i})<k+1$. and let $i_1,\ldots ,i_{k+1}$ be a $k+1$-tuple satisfying the hypotheses above. Define $\vec{j}=\vec{i}\setminus \{i_1\}$, and consider the $m-1$-simplex 
		\[
		\sigma: \Delta^{m-1}\to \Delta^{m}
		\] 
		which skips $i_1$. Then $\sigma\cap \Lambda^{m}_{\vec{i}}=\Lambda^{m-1}_{\vec{j}}$, and so, by the inductive hypothesis, we can fill this simplex to obtain a new simplicial subset 
		\[
		\Lambda^{m}_{\vec{i}}\subset X\subset \Delta^m.
		\]
		We then see that $X$ will consist of precisely those subsimplices of $\Delta^m$ which either (a) skip $i_1$ or (b) skip a vertex $j$ not belonging to $\vec{i}$. More simply put, precisely those simplices which skip a vertex not contained in $\{i_2,\ldots, i_{k+1}\}$. Consequently, 
		\[
		X= (\Lambda^m_{\vec{i}\setminus\{i_1\}},\flat,T_{\vec{i}})
		\]
		and so, by the inductive hypothesis, this map is \bS-anodyne.
	\end{proof}
	
	\begin{remark}
		The proof above actually shows that the map in question is in the saturated hull of morphisms of type \ref{mb:innerhorn}. We will adapt the proof slightly for the next lemma, which will show a related result for morphisms of type  \ref{mb:leftdeglefthorn}. Note that in the first proposition we are adapting \autoref{cons} to a subset $\vec{i}$ containing the vertex $0$.
	\end{remark}
	
	\begin{lemma}\label{lem:indII}
		Let $\vec{i}:=\{0,i_1,\ldots ,i_{k+1}\}$ be a set of distinct vertices of $\Delta^m\coprod_{\Delta^{\{0,1\}}}\Delta^0$  such that 
		\begin{itemize}
			\item $1<i_1<i_2<\cdots <i_{k+1}<n$ 
			\item The $i_j$ are non-consecutive.
			\item For every $1\leq j\leq k+1$, the simplex $\{i_j-1,i_j,i_j+1\}$ is biscaled. 
			\item The simplex $\{0,1,m\}$ is lean-scaled.
		\end{itemize}
		Then the map 
		\[
		(\Lambda_{\vec{i}}^m)\coprod_{\Delta^{\{0,1\}}}\Delta^0\to \Delta^m\coprod_{\Delta^{\{0,1\}}}\Delta^0
		\]
		is \bS-anodyne.
	\end{lemma}
	
	\begin{proof}
		We once again proceed by induction on the length of $\vec{i}$. If $\vec{i}=\{0\}$, then this is a morphism of type \ref{mb:leftdeglefthorn}. If $\vec{i}=\{0,i_1\}$, then we can fill the simplex obtained by deleting $i_1$ using a pushout of type \ref{mb:leftdeglefthorn}, the resulting inclusion is again an inclusion of type \ref{mb:leftdeglefthorn}. 
		
		We now assume, inductively, that the statement holds for any $\vec{i}$ of length less than $k+2$, and let $\vec{i}=\{0,i_1,\ldots, i_{k+1}\}$. Consider the simplex $\sigma:\Delta^{m-1}\to \Delta^m$ obtained by deleting $i_1$. Then we see that 
		\[
		(\Lambda^m_{\vec{i}}\coprod_{\Delta^{\{0,1\}}}\Delta^0)\cap \sigma= \Lambda^{m-1}_{\vec{i}\setminus i_1}\coprod_{\Delta^{\{0,1\}}}\Delta^0
		\]
		so that, by the inductive hypothesis, we can fill $\sigma$ using an \bS-anodyne morphism. The resulting simplicial subset $X$ in 
		\[
		(\Lambda_{\vec{i}}^m)\coprod_{\Delta^{\{0,1\}}}\Delta^0\to X\to  \Delta^m\coprod_{\Delta^{\{0,1\}}}\Delta^0
		\]
		consists of precisely those subsimplices of $\Delta^m$ which skip $i_1$ or which skip an element not in $\vec{i}$. More precisely 
		\[
		X=\Lambda^{m}_{\vec{i}\setminus \{i_1\}}\coprod_{\Delta^{\{0,1\}}} \Delta^0  
		\]
		and thus, by the inductive hypothesis, 
		\[
		X\to \Delta^m\coprod_{\Delta^{\{0,1\}}}\Delta^0 
		\]
		is \bS-anodyne, completing the proof.
	\end{proof}
	
	\begin{proposition}\label{prop:nightmare}
		Denote by  
		\[
		\func{f:(\partial \Delta^n,\flat,\flat)\to (\Delta^n,\flat,\flat)} 
		\]
		a morphism of type \ref{cof:bndry}, and by
		\[ \func{g:\Bigl(\Lambda^m_0\coprod_{\Delta^{\{0,1\}}}\Delta^0,\flat, \flat\subset \{\Delta^{\{0,1,m\}}\}\Bigr)\to \Bigl(\Delta^m\coprod_{\Delta^{\{0,1\}}}\Delta^0,\flat,\flat\subset \{\Delta^{\{0,1,m\}}\}\Bigr)}
		\]
		a morphism of type \ref{mb:leftdeglefthorn}. Then the pushout-product $f\wedge g$ is \bS-anodyne.
	\end{proposition}
	
	Before beginning the proof, we create a diagram for reference. We visualize the product of the targets as a grid, with some simplices which get collapsed. 
	
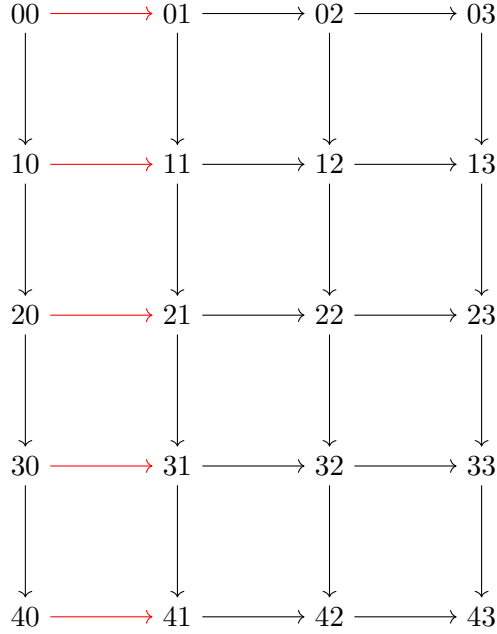
\begin{center}
		\begin{tikzpicture}
		\foreach \x in {0,1,2,3} {
			\foreach \y in {0,1,2,3,4} {
				\pgfmathsetmacro\xr{\x*2}
				\pgfmathsetmacro\yr{(4-\y)*2}
				\path (\xr,\yr) node (A\x\y) {$\y\x$};	
			};
		};
		\foreach \x/\xr in {1/2,2/3}{
			\foreach \y in {0,1,2,3,4}{
				\draw[->] (A\x\y) to (A\xr\y);
			};
		};
		\foreach \y in {0,1,2,3,4}{
			\draw[red,->] (A0\y) to (A1\y);
		};
		\foreach \y/\yr in {0/1,1/2,2/3,3/4}{
			\foreach \x in {0,1,2,3} {
				\draw[->] (A\x\y) to (A\x\yr);
			}; 	
		};
		\end{tikzpicture}
	\end{center}
	In the diagram above, we are looking at $\Delta^4\times \Delta^3$, and the 1-simplices in red are those which get collapsed. 
	
	\begin{proof}
		Since the case $n=0$ is simply the original type \ref{mb:leftdeglefthorn} morphism, we may, without loss of generality, assume $n\geq 1$. We begin by fixing some notation for $n+m$ simplices in $\Delta^n\times \Delta^m$. Note that $n+m$ is the maximal dimension of a non-degenerate simplex. We begin by defining a set $\scr{Z}$ to contain pairs $(\vec{a},\vec{b})$ of vectors with $a_i\in [n]$ and $b_i\in [m]$ and $\on{length}(\vec{a})=\on{length}(\vec{b})$ (including the empty vector). We sometimes will view these pairs as a single string of numbers $(a_1,b_1,\ldots, a_k,b_k)$. The set $\scr{Z}$ will consist precisely of those pairs satisfying the following conditions:
		\begin{itemize}
			\item $a_{i-1}<a_i<n$
			\item $0<b_{i-1}<b_i$ 
		\end{itemize}  
		We can equip $\scr{Z}$ with the structure of a totally ordered by considering a sort of lexicographic order. 
		Setting 
		\[
		A=(a_1,b_1,\ldots, a_k,b_k) \qquad \text{and} \qquad C=(c_1,d_1,\ldots,c_\ell,d_\ell)
		\]
		We say that 
		\begin{itemize}
			\item If there exists a $j$ such that $a_i=c_i$ and $b_i=d_i$ for all $i<j$, and $a_j<c_j$, we say that $A>C$ 
			\item If there exists a $j$ such that $a_i=c_i$ for all $i\leq j$ and $b_i=d_i$ for all $i<j$, and $b_j>d_j$, then $A>C$.
			\item If, for instance, $j>k$, then we set the convention that $a_j=n$ and $b_j=m$. 
		\end{itemize} 
		which yields the desired linear order. So, for instance, we have 
		\[
		(0,1,3,2) < (0,1,2,2) < (0,2,1,4)
		\]
		with our chosen order.
		
		We aim to interpret pairs $(\vec{a},\vec{b})\in \scr{Z}$ as the `turning points' of a $n+m$-simplex in $\Delta^n\times \Delta^m$. We will define a simplex 
		\[
		\func{
			\sigma_{(\vec{a},\vec{b})}:{[n+m]} \to {[n]\times [m]}.
		}
		\]
		Set $k:=\on{length}(\vec{a})=\on{length}(\vec{a})$. We then define $\sigma_{(\vec{a},\vec{b})}$ as a map of posets by 
		\[
		\sigma_{(\vec{a},\vec{b})}(\ell):= \begin{cases}
		\ell 0 & \ell\leq a_1\\
		a_r(\ell-a_r) & a_r+b_{r-1}<\ell\leq a_r+b_r\\
		(\ell-b_r) b_r & a_r+b_r<\ell \leq a_{r+1}+b_r\\
		\end{cases}
		\]
		with the convention that $a_{k+1}=n$ and $b_{k+1}=m$. 
		In the special case of the empty string in $\scr{Z}$, we define
		\[
		\sigma_\emptyset(\ell)=\begin{cases}
		\ell 0 & \ell\leq n \\
		n (\ell-n) & \ell >n
		\end{cases}
		\]
		This can be viewed as a path in the grid above which starts out moving downward, turns right at $a_10$, turns downward at $a_1b_1$, and so on, until it finishes by going right from $n b_k$ to $nm$. Note that if $a_1=0$ the first step is trivial, and if $b_k=m$, the last step is trivial. See \autoref{fig:PathsInproduct} for some examples of vectors and the corresponding paths.
		
		We now define a filtration of $\Delta^n \times (\Delta^m\coprod_{\Delta^{\{0,1\}}} \Delta^0)$ indexed by $\scr{Z}$. Let $\phi$ denote the unique order-preserving bijection 
		\[
		\func{\phi:\{1,2,\ldots, ,|\scr{Z}|\}\to \scr{Z}}.
		\]
		We begin with 
		\[
		X_{0}=\partial\Delta^n\times\left(\Delta^m\coprod_{\Delta^{\{0,1\}}}\Delta^0\right)\coprod_{\partial \Delta^n \times \left(\Lambda^m_0\coprod_{\Delta^{\{0,1\}}}\Delta^0\right)}  \Delta^n\times  \left(\Lambda^m_0\coprod_{\Delta^{\{0,1\}}}\Delta^0\right)
		\]
		And define 
		\[
		X_i := X_{i-1}\cup \sigma_{\phi(i)} \subset \left(\Delta^m\coprod_{\Delta^{\{0,1\}}} \Delta^0\right). 
		\]
		In particular, we have 
		\[
		X_{|\scr{Z}|}=\left(\Delta^m\coprod_{\Delta^{\{0,1\}}} \Delta^0\right).
		\]
		
		We view the $X_i$ for $i>0$ as equipped with the marking and biscaling inherited from $X_{|\scr{Z}|}$. We will aim to show that the inclusions 
		\[
		\func{X_{i-1}\to  X_{i}}
		\]  
		are mb-anodyne. Since both the sources and targets of $f$ and $g$ carry the flat thin-scaling and flat marking, it will suffice for us to only consider the lean triangles. 
		
		We first note that the only simplices of $X_{|\scr{Z}|}$ which are \emph{not} in $X_0$ are those which \emph{both}:
		\begin{itemize}
			\item Contain an element of every column, with the possible exception of the $0$\textsuperscript{th} column. 
			\item Contain an element of every row. 
		\end{itemize}
		To fill these simplices, we will proceed by induction. 
		
		\vspace{1em}
		
		\noindent{\scshape Zeroth case:} The smallest string in $z$ is the empty string. The simplex $\sigma_\emptyset$ traces down the left-hand side of our grid, and then right along the bottom of the grid. The only missing $n+m-1$-simplex in $\sigma_\emptyset$ is subsimplex obtained by forgetting $n0$. However, the simplex $(n-1)0\to n0\to n1$ must be thin-scaled (since it lies in the Cartesian product of $\Delta^1\times \Delta^1$), and thus we obtain a pushout via a $\Lambda^{n+m}_n$-horn of type \ref{mb:innerhorn}. 
		
		\vspace{1em}
		
		\noindent{\scshape First Case:} Now suppose that $\phi(i)=(\vec{a},\vec{b})$ has $a_1\neq 0$. Consider the $n+m-1$ subsimplices of $\sigma_{(\vec{a},\vec{b})}$. Those of these simplices \emph{not} contained in $\sigma_{(\vec{a},\vec{b})}\cap X_{i-1}$ must be obtained from $\sigma_{(\vec{a},\vec{b})}$ by forgetting points of either the form $(a_r,b_{r-1})$ or the form $(a_r,b_r)$. 
		
		If we forget a simplex of the form $(a_r,b_r)$, we obtain a simplex contained in one of the following, previously filled, $n+m$-simplices. (once again, these cases use the convention that $a_{k+1}=n$ and $b_{k+1}=m$)
		When the removed vertex is of the form $(a_r,b_r)$, it is easy to check that the resulting sub-simplex is contained in a previous stage of the induction. Thus, we need only concern ourselves with the removal of vertices of the type $(a_r,b_{r-1})$.
		
		Note that there are precisely $k+1$ subsimplices of $\sigma_{(\vec{a},\vec{b})}$ which arise from forgetting such a vertex: 
		\[
		(a_1,0), (a_2,b_1),\ldots, (a_k,b_{k-1}),(n,b_k)
		\]
		Note, too, that the only other $\sigma_{(\vec{c},\vec{d})}$ containing these will have $(\vec{c},\vec{d})>(\vec{a},\vec{b})$. Each of the 2-simplices $\sigma_{(\vec{a},\vec{b})}(\{j-1,j,j+1\})$ where $\sigma_{(\vec{a},\vec{b})}(j)= (a_\ell,b_{\ell-1})$ will be contained in a copy of $\Delta^1\times \Delta^1$, and thus be biscaled by the product scaling. Thus, by \autoref{lem:indI}, we can add $\sigma_{(\vec{a},\vec{b})}$ with an \bS-anodyne pushout. Thus, $X_{i-1}\to X_i$ is \bS-anodyne.
		
		\vspace{1em} 
		\noindent{\scshape Second Case:}
		We now suppose that $\phi(i)=(\vec{a},\vec{b})$ has $a_1=0$. Considering the $n+m-1$-subsimplices of $\sigma_{(\vec{a},\vec{b})}$, we again see that those simplices not contained in $\sigma_{(\vec{a},\vec{b})}\cap X_{i-1}$ are obtained by forgetting the vertex $(0,0)$, a vertex of the form $(a_r,b_r)$, or a vertex of the form $(a_r,b_{r-1})$. 
		
		As in the first case, we see that those obtained by forgetting $(a_r,b_r)$ are in $\sigma_{(\vec{c},\vec{d})}$ for some $(\vec{c},\vec{d})<(\vec{a},\vec{b})$, and those of the other two forms are not. Note that, since $a_1=0$, the simplex $\sigma_{(\vec{a},\vec{b})}(0,1,n+m)$ is lean scaled (since $\sigma_{(\vec{a},\vec{b})}(0)=\sigma_{(\vec{a},\vec{b})}(1)$ once projected down to $\Delta^n$). Similarly, each of the 2-simplices $\sigma_{(\vec{a},\vec{b})}(\{j-1,j,j+1\})$ where $\sigma_{(\vec{a},\vec{b})}(j)= (a_\ell,b_{\ell-1})$ will be contained in a copy of $\Delta^1\times \Delta^1$, and thus be biscaled by the product scaling. Applying \autoref{lem:indII} then shows that we can fill $\sigma_{(\vec{a},\vec{b})}$ with an \bS-anodyne pushout, and so 
		\[
		\func{X_i\to X_{i+1}}
		\]
		is mb-anodyne. 
		
		The proof is now complete, as $f\wedge g$ is the composite of the morphisms 
		\[
		X_0\to X_1\to X_2\to \cdots \to X_{|\scr{Z}|},
		\]
		and thus, is \bS-anodyne.
	\end{proof}
	
	While a significant majority of the cases of the pushout-product remain, all of the remaining cases involve far less difficulty than this one. We can now turn to the main event.
	
	\begin{figure}
		\begin{center}
			\begin{tikzpicture}
			\foreach \x in {0,1,2,3} {
				\foreach \y in {0,1,2,3,4} {
					\pgfmathsetmacro\xr{\x*2}
					\pgfmathsetmacro\yr{(4-\y)*2}
					\path (\xr,\yr) node (A\x\y) {$\y\x$};	
				};
			};
			\foreach \x/\xr in {0/1,1/2}{
				\foreach \y in {0,1,2,3,4}{
					\draw[->] (A\x\y) to (A\xr\y);
				};
			};
			\foreach \y in {0,1,2,3,4}{
				\draw[->] (A2\y) to (A3\y);
			};
			\foreach \y/\yr in {0/1,1/2,2/3,3/4}{
				\foreach \x in {0,1,2,3} {
					\draw[->] (A\x\y) to (A\x\yr);
				}; 	
			};
			\draw[thick,blue] (A00) to (A10) to (A11) to (A12) to (A22) to (A32)to (A33) to (A34); 
			\draw[thick,red] (A00) to (A01) to (A11) to (A21)to (A22) to (A23) to (A24) to (A34);
			\end{tikzpicture}
		\end{center}
		
		\caption{Above, we depict in {\color{blue}blue} the simplex corresponding to the vectors $\vec{a}=(0,2)$ and $\vec{b}=(1,3)$ or, equivalently, to the string $(0,1,2,3)$. In {\color{red}red}, we give the simplex corresponding to the vectors $\vec{a}=(1)$ and $\vec{b}=(2)$ or, equivalently, to the string $(1,2)$. Note that, according to the ordering on $\scr{Z}$, we have $(1,2)<(0,1,2,3)$ (because 0<1).}
	\end{figure}
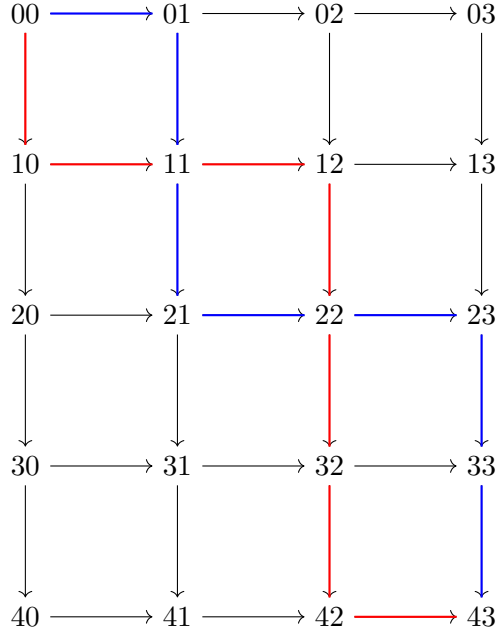
	
	\begin{proof}[Proof (of \autoref{prop:PP})]
		The proof will consist of the usual rigmarole --- checking on pairs of generators. While there are 44 cases in all, the vast majority of these turn out to be trivial or extremely simple. The two cases dealt with by the preceding propositions are by far the most complicated cases. 
		
		We will label our cases first by the generating cofibration, and then by the generating \bS-anodyne morphism. 
		\begin{itemize}
			\item[\ref{cof:bndry}] The cofibration is of the form $\Bigr(\partial \Delta^n,\flat,\flat\Bigl) \rightarrow \Bigr( \Delta^n,\flat,\flat\Bigl)$.
			\begin{itemize}
				\item[\ref{mb:innerhorn}] Since the marking is trivial, and the thin and lean scalings agree, we can consider only the thin scalings. Case (1A) from 3.1.8 in \cite{LurieGoodwillie} then shows that this can be obtained as a pushout of morphisms of type \ref{mb:innerhorn} and morphisms for the type from remark 3.1.4 in \cite{LurieGoodwillie}. 
				\item[\ref{mb:wonky4}] This is precisely case (1B) from 3.1.8 \cite{LurieGoodwillie} 
				\item[\ref{mb:leftdeglefthorn}] This is \autoref{prop:nightmare}.
				\item[\ref{mb:2Cartesianmorphs}]  The dual of the argument given for \autoref{prop:nightmare} suffices once we have replaced "degenerate 1-simplices" with "marked 1-simplices".
				\item[\ref{mb:2CartliftsExist}] We note that the map of underlying simplicial sets is 
				\[
				Y_0:= (\Delta^n \times \{1\})\coprod_{\partial \Delta^n \times \{1\}} (\partial\Delta^n\times \Delta^1) \to \Delta^n \times \Delta^1 
				\]
				We can define a sequence of $n+1$ simplices in $\Delta^n\times \Delta^1$ via the maps 
				\[
				\func*{
					\sigma_k: {[n+1]}\to {[n]\times [1]};
					i \mapsto \begin{cases}
					(i,0) & i\leq k\\
					(i-1,1) & i>k
					\end{cases}
				}
				\]
				We then define $Y_i$ inductively as $Y_{i-1}\cup \sigma_{i-1}$ (Following HTT 2.1.2.6). We see that the morphism $\func{Y_{i-1}\to Y_i}$ is a pushout with a $\Lambda^{n+1}_{i+1}$-horn. It will thus suffice for us to note two things: 
				\begin{itemize}
					\item When $i<n$, the 2-simplex $\sigma_{i}|_{\Delta^{\{i-1,i,i+1\}}}$ is the simplex 
					\[
					(i,0) \to (i,1) \to (i+1,1)
					\]
					in $\Delta^{\{i-1,i\}}\times \Delta^1$, and thus is necessarily thin-scaled. We thus obtain a pushout of type \ref{mb:innerhorn}. 
					\item when $i=n$, the 2-simplex $\sigma_{n}|_{\Delta^{\{0,n-1,n\}}}$ is 
					the simplex 
					\[
					(0,0) \to (n,0) \to (n,1)
					\]
					in $\Delta^{\{0,n\}}\times \Delta^1$, and thus is necessarily thin-scaled. Moreover, the morphism $\sigma_{n+1}|_{\Delta^{\{n-1,n\}}}$ is 
					\[
					(n,0)\to (n,1)
					\]
					and thus is marked. Hence, we obtain a pushout of type \ref{mb:2Cartesianmorphs}.
				\end{itemize} 
				\item[\ref{mb:composeacrossthin}] This is an isomorphism when $n\geq 1$, and is a morphism of type \ref{mb:composeacrossthin} when $n=0$.
				\item[\ref{mb:coCartoverThin}] This is an isomorphism on underlying marked lean scaled simplicial sets, and thus in the saturated hull of morphisms of type \ref{mb:coCartoverThin}.
				\item[\ref{mb:innersaturation}] We will treat the case $i=2$ --- the case $i=1$ follows virtually identically. When $n>2$ this is an isomorphism and when $n=0$, this is a morphism of type \ref{mb:innersaturation}. This means that we may consider the following two cases: 
				\begin{itemize}
					\item If $n=2$, we note that this is an isomorphism on the underlying marked simplicial sets, and indeed differs only in the lean-scaling. The only missing lean-scaled simplex is $00\to 11\to 23$ in $\Delta^2\times \Delta^3$. We may expand this to a 3-simplex $00\to 11\to 12\to 23$. It is easily checked that this 3-simplex gives us a pushout of type \ref{mb:innersaturation} (with $i=1$), showing that the morphism is \bS-anodyne.
					\item If $n=1$, we again have that the source and target differ only in their lean-scaling. It is easy to check that the missing simplices are the simplices $00\to 11\to 13$ and $00\to 01\to 13$ in $\Delta^1\times \Delta^3$. In the former case, we can extend to the 3-simplex $00\to 11\to 12\to 13$ and scale the desired 2-simplex with a pushout of type \ref{mb:innersaturation}, and in the latter case we can extend to the 3-simplex $00\to 01\to 02\to 13$ and scaled the desired 2-simplex with a pushout of type \ref{mb:innersaturation}.   
				\end{itemize}
				\item[\ref{mb:dualcocart2of3}] This is formally dual to the next case.
				\item[\ref{mb:coCart2of3}] When $n\geq 2$, this is an isomorphism. When $n=0$, this is a morphism of type \ref{mb:coCart2of3}. When $n=1$, we get the identity on underlying marked simplicial sets 
				\[
				(\Delta^3)^\dagger\times(\Delta^1)^\flat  \to (\Delta^3)^\dagger\times(\Delta^1)^\flat 
				\]
				The lean scaling on the target is maximal. The missing scaled simplices in the source are $00\to 10\to 21$, $00\to 11\to 21$. One can then note that the 3-simplex $00\to 11\to 21\to 31$ is of type \ref{mb:coCart2of3}, and can thus be filled. Similarly, the 3-simplex $00\to 10\to 21\to 31$ is of type \ref{mb:coCart2of3}, and can be filled. 
				\item[\ref{mb:equivalences}] If $n\geq 1$, this is an isomorphism. If $n=0$, this is again a morphism of type \ref{mb:equivalences}.
			\end{itemize}
			\item[\ref{cof:marked1}] The cofibration is of the form $\Bigr(\Delta^1,\flat,\flat \Bigl) \rightarrow \Bigr(\Delta^1,\sharp,\flat \Bigl) $.
			\begin{itemize}
				\item[\ref{mb:innerhorn}]  This is isomorphism on underlying marked, lean-scaled simplicial sets, and thus \bS-anodyne. 
				\item[\ref{mb:wonky4}] This is an isomorphism.
				\item[\ref{mb:leftdeglefthorn}] This is an isomorphism. 
				\item[\ref{mb:2Cartesianmorphs}] This is an isomorphism. 
				\item[\ref{mb:2CartliftsExist}] This gives us the inclusion 
				\[
				\func{
					(\Delta^1\times \Delta^1)^\dagger_{\sharp\subset \sharp} \to (\Delta^1\times \Delta^1)^\sharp_{\sharp\subset\sharp}
				}
				\] 
				Where $\dagger$ is the marking containing $\Delta^1\times \{0\}$, $\Delta^1\times \{1\}$, and $\{1\}\times \Delta^1$. A pushout of type \ref{mb:composeacrossthin} marks the diagonal, and a pushout by the morphism
				\[
				\Bigr(\Delta^2, \{\Delta^{\{1,2\}},\Delta^{\{0,2\}}\},\sharp\Bigl) \to \Bigr(\Delta^2,\sharp,\sharp\Bigl)
				\]
				marks the remaining edge. By  \autoref{lem:2Cart2of3}, this is \bS-anodyne.
				\item[\ref{mb:composeacrossthin}] This is the identity on  $(\Delta^2\times \Delta^1)_\sharp$ on underlying biscaled simplicial sets. The only 1-simplex which is not marked in the source is $00\to 21$, and the target is maximally marked. We can add the remaining marked edge using a pushout of type \ref{mb:composeacrossthin}.
				\item[\ref{mb:coCartoverThin}] This is an isomorphism. 
				\item[\ref{mb:innersaturation}] This is an isomorphism.
				\item[\ref{mb:dualcocart2of3}] This is an isomorphism. 
				\item[\ref{mb:coCart2of3}] This is an isomorphism. 
				\item[\ref{mb:equivalences}] The source and target of the pushout-product differ only in their marking. However, every edge which is marked in the target by not the source will be the product of a non-degenerate edge in $K$ and the non-degenerate edge in $\Delta^1$. Consequently, it will be the diagonal in a square $\Delta^1\times \Delta^1\subset \Delta^1\times K$. Since every other 1-simplex of this square will be marked, the diagonal can be marked with a pushout of type \ref{mb:composeacrossthin}.
			\end{itemize}
			\item[\ref{cof:coCart}] The cofibration is of the form $\Bigr(\Delta^2,\flat,\flat\Bigl) \rightarrow \Bigr(\Delta^2,\flat,\flat \subset \sharp)$.
			\begin{itemize}
				\item[\ref{mb:innerhorn}] When $n>2$, this is an isomorphism. If $n=2$, this is an isomorphism on the underlying marked thin-scaled simplicial sets, so we can consider only the lean scaling. 
				
				The target is maximally lean scaled. In the source, there are precisely three 2-simplices which are not lean scaled: 
				\begin{align}
				00 &\to 12\to 22\\
				00 &\to 11 \to 22\\
				00 & \to 10 \to 22
				\end{align}
				For the first, we can extend to the 3-simplex $00\to 02\to 12\to 22$, and obtain obtain a pushout of type \ref{mb:innersaturation} with $i=1$. For the third, we can extend to the 3-simplex $00\to 10\to 20\to 22$, and obtain a pushout of type \ref{mb:innersaturation} with $i=2$. For the second, we can then extend to the 3-simplex $00\to 10\to 11\to 22$, and obtain a pushout of type \ref{mb:innersaturation} (with $i=1$).
				\item[\ref{mb:wonky4}] The pushout-product is an isomorphism on underlying marked thin-scaled simplicial sets, so once again we consider the lean triangles. The underlying simplicial sets are both $\Delta^2\times \Delta^4$. There are two triangles which are lean in the target, but not the source, namely:
				\begin{align}
				00 & \to 13\to 24\\
				00 & \to 11 \to 24
				\end{align}
				For (4), if we extend to the 3-simplex $00\to 03\to 13\to 24$, we obtain a pushout of type \ref{mb:innersaturation} with $i=1$. For (5), if we extend to the 3-simplex $00\to 11\to 21 \to 24$, we obtain a pushout of type \ref{mb:innersaturation} with $i=2$. 
				\item[\ref{mb:leftdeglefthorn}] This is an isomorphism when $n>2$. When $n=2$, we first note that we can neglect the thin scaling and the marking. Since this is the case, we consider the corresponding inclusion of lean-scaled simplicial sets. The underlying map is 
				\[
				\func{\id: \Delta^2\times (\Delta^2\coprod\Delta^0) \to \Delta^2\times (\Delta^2\coprod\Delta^0) }
				\]
				and the target carries a maximal scaling. The only unscaled simplex in the source is 
				\[
				\func{00\to 11 \to 22}
				\] 
				We can then consider the simplex 
				\[
				00 \to 01\to 11 \to 22
				\]
				Since $00\to 01$ is degenerate, we can scale the remaining simplex via a pushout of type \ref{mb:dualcocart2of3}. 
				\item[\ref{mb:2Cartesianmorphs}] This is an isomorphism when $n>2$. When $n=2$, we again note that is sufficient only to consider the marking and the lean scaling. In this case, we obtain an isomorphism on the underlying simplicial set $\Delta^2\times \Delta^2$. The markings are identical on the source and target, so we are again left to consider only the lean scaling. The target is maximally scaled, and the only unscaled simplex in the source is $00\to 11\to 22$. Considering the 3-simplex 
				\[
				00 \to 11 \to 21\to 22,
				\]  
				we note that $21\to 22$ is marked. Thus, a pushout of type \ref{mb:coCart2of3} suffices.
				\item[\ref{mb:2CartliftsExist}] 
				The underlying map of simplicial sets is the identity on $\Delta^2\times \Delta^1$. It is, as above, and isomorphism on the marking and thin-scaling. There are precisely three simplices which we need to lean-scale: 
				\begin{align}
				00&\to 11\to 21\\
				00&\to 10\to 21 \\
				00&\to 10\to 20
				\end{align}
				For (6), we can extend to the 3-simplex $00\to 01\to 11\to 21$, and then obtain the desired scaling via a pushout of type \ref{mb:innersaturation} with $i=1$. For (7), we can extend to the 3-simplex $00\to 10\to 11 \to 21$ and obtain the desired scaling via a pushout of type \ref{mb:innersaturation} with $i=2$. Finally, for (8), we can extend to the 3-simplex $00\to 10\to 20\to 21$, and obtain a pushout of type \ref{mb:coCart2of3} (since the morphism $20\to 21$ is marked).
				
				\item[\ref{mb:composeacrossthin}] This is an isomorphism.
				\item[\ref{mb:coCartoverThin}] This is an isomorphism on the underlying marked lean-scaled simplicial sets, and thus a sequence of pushouts of type \ref{mb:coCartoverThin}.
				\item[\ref{mb:innersaturation}] In both cases, the underlying map of simplicial sets is the identity on $\Delta^2\times \Delta^3$, and in both cases, there is only one 2-simplex we need to lean scale.
				\begin{itemize}
					\item When $i=2$, the missing scaling is on $00\to 11\to 23$. We can extend to the 3-simplex $00\to 11\to 21\to 23$, and scale the missing 2-simplex using a pushout of type \ref{mb:innersaturation} with $i=2$. 
					\item When $i=1$, the missing scaling is on $00\to 12\to 23$. We can extend to the 3-simplex $00\to 02\to 12\to 23$, and scale the missing 2-simplex using a pushout of type \ref{mb:innersaturation} with $i=1$.
				\end{itemize}
				\item[\ref{mb:dualcocart2of3}] This is effectively dual to the next case.
				\item[\ref{mb:coCart2of3}] On underlying marked simplicial sets, this is the identity on the marked simplicial set 
				\[
				(\Delta^3,\{\Delta^{\{2,3\}}\}) \times (\Delta^2)^\flat
				\]
				The only simplex which is lean-scaled in the target but not the source is $00\to 11\to 22$. However, if we consider the 3-simplex 
				\[
				00\to 11\to 22\to 32
				\]
				in $\Delta^3\times \Delta^2$, we obtain a pushout of type \ref{mb:coCart2of3} giving the desired scaling.
				\item[\ref{mb:equivalences}] This is an isomorphism. 
			\end{itemize}
			\item[\ref{cof:thin}] The cofibration is of the form $\Bigl(\Delta^2,\flat,\flat \subset \sharp \Bigr) \rightarrow \Bigl( \Delta^2,\flat,\sharp\Bigr)$.
			\begin{itemize}
				\item[\ref{mb:innerhorn}-\ref{mb:equivalences}] All of these are, necessarily, isomorphisms on the underlying marked lean-scaled simplicial sets (since, forgetting about thin simplices, the morphisms of type 4 are isomorphisms of marked lean-scaled simplicial sets), and thus are \bS-anodyne.  \qedhere 
			\end{itemize}
		\end{itemize}
	\end{proof}
	
	Though the preceding arguments may seem an abuse of the reader's patience, now that the pushout-product is established, we can freely use it without directly working with these technicalities. In particular, we gain access to well-behaved mapping spaces, mapping categories, and mapping bicategories for $(\mbsSet)_{/S}$ --- a key convenience in the work to come.

\begin{definition}\label{def:mappingbicats}
  Given two mb simplicial sets $(K,E_K,T_K \subseteq C_K), (X,E_X,T_X \subseteq C_X)$ we define  another mb simplicial set denoted by $\on{Fun}^{\mathbf{mb}}(K,X)$ and characterized by the following universal property
  \[
    \on{Hom}_{\mbsSet}\Bigr(A,\on{Fun}^{\mathbf{mb}}(K,X) \Bigl)\isom \Hom_{\mbsSet}\Bigr(A \times K,X  \Bigl).
  \]
\end{definition}

As a direct consequence of \autoref{prop:PP} we obtain the following corollary.

\begin{corollary}\label{cor:bsfibfun}
  Let $f:(X,E_X,T_X \subseteq C_X) \to (Y,E_Y,T_Y \subseteq C_Y)$ be a \bS-fibration. Then for every $K \in \mbsSet$ the induced morphism $\on{Fun}^{\mathbf{mb}}(K,X) \to \on{Fun}^{\mathbf{mb}}(K,Y)$ is a \bS-fibration.
\end{corollary}

\begin{definition}
  Let $f: X \to Y$ be a \bS-fibration and consider another map of mb simplicial sets $g:K \to Y$. We define an $\infty$-bicategory $\on{Map}_Y(K,X)$ by means of the pullback square
  \[
    \begin{tikzcd}[ampersand replacement=\&]
      \on{Map}_Y(K,X) \arrow[r] \arrow[d] \& \on{Fun}^{\mathbf{mb}}(K,X) \arrow[d] \\
      \Delta^0 \arrow[r,"g"] \& \on{Fun}^{\mathbf{mb}}(K,Y)
    \end{tikzcd}
  \]
\end{definition}

\begin{proposition}\label{prop:mappingcofbs}
  Let $f: X \to Y$ be a \bS-fibration. Suppose that we are given morphisms of mb simplicial sets 
  \[
    \func{L \to[h] K \to[g] Y}
  \]
  such that $h$ is a cofibration (resp. \bS-anodyne). Then the induced morphism
  \[
    \func{h^*:\on{Map}_Y(K,X) \to \on{Map}_Y(L,X)}
  \]
  is a fibration of scaled simplicial sets (resp. trivial fibration).
\end{proposition}
\begin{proof}
  Suppose that $h$ is a cofibration. Then it follows from \autoref{prop:PP} that $h^{*}$ has the right lifting property against the class of scaled anodyne maps. Note that according to \autoref{lem:overthepoint} the marking on both $\infty$-bicategories is precisely given by equivalences. Therefore using \ref{mb:2Cartesianmorphs} in \autoref{def:mbsanodyne} for $n=1$ we see that $h^{*}$ is an isofibration. We can conclude from the construction of the model structure on $\Set_\Delta^{\mathbf{sc}}$ as a Cisinski model structure in \cite{GHL_equiv} that $h^{*}$ is a fibration of $\infty$-bicategories. The case where $h$ is a \bS-anodyne follows immediately from \autoref{prop:PP}.
\end{proof}
  
\subsection{The model structure}\label{subsec:MS}
Let $S \in \on{Set}^{\on{sc}}_{\Delta}$ for the rest of the section we will denote $(\mbsSet)_{/S}$ the category of mb simplicial set over $(S,\sharp,T_S \subset \sharp)$.

\begin{definition}\label{def:fibrantobjects}
  We say that an object $\pi:X \to S$ in $(\mbsSet)_{/S}$ is an \emph{outer} 2-\emph{Cartesian} fibration if it is a $\bS$-fibration.
\end{definition}

\begin{remark}
 We will frequently abuse notation and refer to outer 2-Cartesian as \emph{2-Cartesian fibrations}.
\end{remark}

\begin{definition}\label{def:underlyingmapping}
  Let $\pi:X \to S$ be a morphism of mb simplicial sets. Given an object $K\to S$, we define $\on{Map}^{\on{th}}_{S}(K,X)$ to be the  mb sub-simplicial set consisting only of the thin triangles. Note that if $\pi$ is a 2-Cartesian fibration this is precisely the underlying $\infty$-category of $\on{Map}_S(K,X)$. 
  
  We similarly denote by $\on{Map}^{\isom}_S(K,X)$ the mb sub-simplicial set consisting of thin triangles and marked edges. As before, we note that if $\pi$ is a 2-Cartesian fibration, the simplicial set $\on{Map}^{\isom}_S(K,X)$ can be identified with the maximal Kan complex in $\on{Map}_S(K,X)$.
\end{definition}

\begin{definition}
  We define a functor $\func{I: \on{Set}^+_{\Delta} \to \mbsSet}$ mapping a marked simplicial set $(K,E_K)$ to the mb simplicial set $(K,E_K,\sharp)$. If $K$ is maximally marked we adopt the notation $I^{+}(K^{\sharp})=K^{\sharp}_{\sharp}$
\end{definition}

\begin{remark}
  Note that we can endow the $(\mbsSet)_{/S}$ with the structure of a $\on{Set}_{\Delta}^{+}$-enriched category by means of $\on{Map}^{\on{th}}_{S}(\mathblank,\mathblank)$. In addition given $K \in \on{Set}_{\Delta}^+$ and $\pi:X \to S$ we define  $K \tensor X:= I(K) \times X$ equipped with a map to $S$ given by first projecting to $X$ and then composing with $\pi$. This construction shows that $(\mbsSet)_{/S}$ is tensored over $\on{Set}^+_{\Delta}$. One can easily show that $(\mbsSet)_{/S}$ is also cotensored over $\on{Set}^+_{\Delta}$.

  In a similar way one can use $\on{Map}^{\isom}_{S}(\mathblank,\mathblank)$ to endow $(\mbsSet)_{/S}$ with the structure of a $\on{Set}_{\Delta}$-enriched category. In this case the cotensor is given by $K \tensor X= I(K^{\sharp}) \times X$.
\end{remark}

\begin{definition}\label{def:weakequiv1}
  Let $\func{L \to[h] K \to[p] S}$ be a morphism in $(\mbsSet)_{/S}$. We say that $h$ is a cofibration when it is a monomorphism of simplicial sets. We will call $h$ a weak equivalence if for every 2-Cartesian fibration $\pi:X \to S$ the induced morphism
  \[
    \func{h^{*}:\on{Map}_S(K,X) \to \on{Map}_S(L,X)}
  \]
  is a bicategorical equivalence.
\end{definition}

\begin{definition}
	Given two mb simplicial sets $p:X\to S$ and $q:Y\to S$ over $S$, we call a morphism 
	\[
	\begin{tikzcd}
	(\Delta^1,\sharp,\sharp)\times X \arrow[rr,"h"] \arrow[rd,"p"'] &   & Y \arrow[ld,"q"] \\
	& S &             
	\end{tikzcd}
	\]
	a \emph{marked homotopy over $S$} from $h|_{\{0\}\times X}$ to $h|_{\{1\}\times X}$. We say that a morphism $f:X\to Y$ is a \emph{marked homotopy equivalence} if there is a morphism $g:Y\to X$ over $S$ and  marked homotopies from $f\circ g$ to $\id_Y$ and from $g\circ f$ to $\id_X$. 
\end{definition}

\begin{proposition}\label{prop:leftproper}
  Suppose we are given a pushout diagram in $(\mbsSet)_{/S}$
  \[
    \begin{tikzcd}[ampersand replacement=\&]
      L \arrow[d,"u"] \arrow[r,"v"] \& K \arrow[d] \\
      R \arrow[r,"w"] \& P
    \end{tikzcd}
  \]
where $u$ is a cofibration and $v$ is a weak equivalence. Then $w$ is also a weak equivalence.
\end{proposition}
\begin{proof}
  Let $\pi:X \to Y$ be a 2-Cartesian fibration. Then it follows that we have a pullback diagram of fibrant scaled simplicial sets
  \[
    \begin{tikzcd}[ampersand replacement=\&]
      \on{Map}_S(P,X) \arrow[r,"w^*"] \arrow[d] \& \on{Map}_S(R,X) \arrow[d,"u^*"] \\
      \on{Map}_S(K,X) \arrow[r,"v^*"] \& \on{Map}_S(L,X)
    \end{tikzcd}
  \]
  where $u^*$ is a fibration according to \autoref{prop:mappingcofbs} and $v^*$ is a bicategorical equivalence. Since this pullback already represents the homotopy pullback it follows that $w^{*}$ is also a bicategorical equivalence.
\end{proof}

\begin{proposition}\label{prop:weakequivkan}
  Let $\func{L \to[h] K \to[p] S}$ be a morphism in $(\mbsSet)_{/S}$. Then the following are equivalent
  \begin{itemize}
     \myitem{i)}\label{equivs:WE} The map $h:L \to K$ is a weak equivalence.
     \myitem{ii)}\label{equivs:inftycat} For every 2-Cartesian fibration $\pi:X \to S$ the induced morphism
     \[
       \func{\on{Map}^{\on{th}}_S(K,X) \to[\isom] \on{Map}^{\on{th}}_S(L,X)}
     \]
     is an equivalence of $\infty$-categories.
     \myitem{iii)}\label{equivs:space}  For every 2-Cartesian fibration $\pi:X \to S$ the induced morphism
     \[
       \func{\on{Map}^{\isom}_S(K,X) \to[\isom] \on{Map}^{\isom}_S(L,X)}
     \]
     is a homotopy equivalence of Kan complexes.
   \end{itemize} 
\end{proposition}
\begin{proof}
 Using the small object argument we can factor the morphism $p$ (resp. $q=p \circ h$) 
  \[
    \func{K \to F_K \to S}
  \]
  where the first morphism is \bS-anodyne and the second is a 2-Cartesian fibration and similarly for $q$. Using \autoref{prop:mappingcofbs} we obtain for every 2-Cartesian fibration $\pi:X \to S$ a commutative diagram
  \[
    \begin{tikzcd}[ampersand replacement=\&]
      \on{Map}_S(F_K,X) \arrow[r] \arrow[d,"\isom"] \& \on{Map}_S(F_L,X) \arrow[d,"\isom"] \\
      \on{Map}_S(K,X)  \arrow[r] \& \on{Map}_S(L,X) 
    \end{tikzcd}
  \]
  where the horizontal morphisms are trivial fibrations of $\infty$-bicategories. Consequently, we can assume that both $p$ and $q$ are 2-Cartesian fibrations. The implications \ref{equivs:WE}$\implies$ \ref{equivs:inftycat}$\implies$ \ref{equivs:space} are obvious. Now let us assume that \ref{equivs:space} holds. Since we are assuming that both $p,q$ are 2-Cartesian fibrations we obtain a homotopy equivalence of Kan complexes
  \[
    \func{\on{Map}^{\isom}_S(K,L) \to[\isom] \on{Map}^{\isom}_S(L,L)}
  \]
  It follows that we have a morphism $\gamma:K \to L$ over $S$ and a homotopy (again over $S$) expressing $\gamma \circ h \sim \on{id}_q$. Observe that both  $h \circ \gamma$ and $\on{id}_p$ get mapped under
  \[
    \func{\on{Map}^{\isom}_S(K,K) \to[\isom] \on{Map}^{\isom}_S(L,K)}
  \]
  to equivalent objects. Using our hypothesis it follows that $h \circ \gamma \sim \on{id}_p$. The claim now follows easily.
\end{proof}

\begin{lemma}
	Let $\func{L \to[h] K \to[p] S}$ be a morphism in $(\mbsSet)_{/S}$ such that $p:K\to S$ and $p\circ h: L\to S$ are 2-Cartesian fibrations. Then the conditions \ref{equivs:WE}-\ref{equivs:space} in \autoref{prop:weakequivkan} are additionally equivalent to
	\begin{itemize}
		\myitem{iv)}\label{equiv:hmtpy} The morphism $f$ is a marked homotopy equivalence over $S$.  
	\end{itemize}
\end{lemma}

\begin{proof}
	 The equivalence of \ref{equiv:hmtpy} and \ref{equivs:space} is purely formal, so the result follows from \autoref{prop:weakequivkan}
\end{proof}

\begin{definition}
  We say that a morphism $\func{L \to[h] K \to S}$ is a trivial fibration it it has the right lifting property against the class of cofibrations.
\end{definition}

\begin{remark}
  Observe that every trivial fibration is in particular a weak equivalence.
\end{remark}

\begin{proposition}\label{prop:trivfibfibres}
  Given a diagram of the form
  \[
    \begin{tikzcd}
     X \arrow[rr,"f"] \arrow[rd,"p"] &   & Y \arrow[ld,swap,"q"] \\
                        & S &             
    \end{tikzcd}
  \]
  where both $p$ and $q$ are 2-Cartesian fibrations. Then the following statements are equivalent:
  \begin{itemize}
    \item[i)] The map $f$ is a trivial fibration.
    \myitem{ii)}\label{item:forExtravaganza} The map $f$ has the right lifting property against \bS-anodyne maps and for every $s \in S$ the induced map on fibres $\func{f_s:X_s \to[\isom] Y_s}$ is a bicategorical equivalence.
  \end{itemize}
\end{proposition}
\begin{proof}
  The implication $i) \implies ii)$ is clear. Now suppose that $ii)$ holds. Then we immediately see that for every $s \in S$ the map $f_s$ is a trivial fibration of scaled simplicial sets. First we will show that we can lift the maps 
  \[
    \func{ \Bigr(\partial \Delta^n,\flat,\flat\Bigl) \to  \Bigr( \Delta^n,\flat,\flat\Bigl)} \quad, \quad n\geq 0.
  \]
  Suppose we are given a lifting problem of the form
  \[
    \begin{tikzcd}[ampersand replacement=\&]
      \partial \Delta^n \arrow[r,"\alpha"] \arrow[d] \& X \arrow[d,"f"] \\
      \Delta^n \arrow[r,"\beta"] \& Y
    \end{tikzcd}
  \]
  and let $\kappa_{\beta}$ be the smallest integer such that $q \circ \beta: \Delta^n \to \Delta^{\kappa_{\beta}} \to S$. We will use induction on $\kappa_{\beta}$. Note that when $\kappa_{\beta}=0$ the lifting problem occurs in one of the fibres and thus the solution exists. Suppose claim holds for $0<\kappa_{\beta} -1\leq n-1$. We will assume without loss of generality that $S=\Delta^{\kappa_{\beta}}$. Let us remark that by construction the map $r=q \circ \beta:\Delta^n \to \Delta^{\kappa_{\beta}}$ must be surjective. Let $j_{\beta}\in [n] $ be the biggest element such that $r(j_\beta) <r(n)=\kappa_{\beta}$. We can now use \autoref{lem:inductionextravaganzza} to produce a commutative diagram 
  \[
    \begin{tikzcd}[ampersand replacement=\&]
      S^{m+1}_{\vec{j}} \arrow[r,"\epsilon"] \arrow[d] \& X \arrow[d,"f"] \\
      \Delta^{n+1} \arrow[r,"\theta"] \& Y
    \end{tikzcd}
  \]
  satisfying the conditions of the lemma. It follows from the proof \autoref{lem:inductionextravaganzza} that the triangle $\theta(\set{j_{\beta},j_{\beta}+1,j_{\beta}+2})$ must be scaled. Restricting this diagram along the face missing $j_{\beta}+2$ yields another commutative square
  \[
    \begin{tikzcd}[ampersand replacement=\&]
      \partial \Delta^n \arrow[r] \arrow[d] \& X \arrow[d,"f"] \\
      \Delta^{n} \arrow[r,"\xi"] \& Y
    \end{tikzcd}
  \]
  It is immediate to see that our original lifting problem admits a solution if this later lifting problem admits a solution. We can further see that if $j_{\beta}=n-1$ then $\xi$ must factor through $\Delta^{\kappa_{\beta}-1}$ and the existence of the solution follows from the inductive hypothesis. If $j_{\beta}<n-1$ it follows that $q \circ \xi$ must be surjective and that $j_{\xi}>j_{\beta}$ so we can keep applying \autoref{lem:inductionextravaganzza} until we  obtain the solution. The inductive step is proved and the claim holds.

  To finish the proof we must show that $f$ detects marked edges and lean (resp. thin) triangles. Let $e:\Delta^1 \to X$ such that $f(e)$ is marked. Let us denote $e(i)=x_i$ for $i \in \set{0,1}$ and similarly denote $f(x_i)=y_i$. Pick a marked lift $\widetilde{e}:\hat{x}_0 \to x_1$ and observe that we can produce a 2-simplex $\sigma: \Delta^2 \to X$ such that $\restr{\sigma}{\Delta^{\set{1,2} }}=\widetilde{e}$ and $\restr{\sigma}{\Delta^{\set{0,2} }}=e$. It follows that $f(\sigma)$ is fully marked and since its restriction to $\Delta^{\set{0,1}}$ lies in $Y_{q(y_0)}$ that particular edge must be an equivalence. However $f$ detects equivalences in the fibres so it follows that $\restr{\sigma}{\Delta^{\set{0,1}}}$ is marked in $X$. The claim follows from \autoref{def:mbsanodyne} \ref{mb:composeacrossthin}.

  Suppose we are given $\varphi:\Delta^2 \to X$ such that $f(\varphi)$ is thin-scaled in $Y$. As usual we will assume without loss of generality that $S=\Delta^2_{\sharp}$ a maximally scaled 2-simplex. We can additionally assume that $\varphi$ is not contained in some $X_i$ for $i \in [2]$, otherwise the claim follows immediately. Let $s=p \circ \varphi$ and assume that $s$ factors through some $\Delta^1$. We define $j_{\varphi}$ as the biggest integer such that $s(j_{\varphi})=0 < s(2)$. Then a totally analogous argument to that of \autoref{lem:inductionextravaganzza} shows that we can produce a 3-simplex $T:\Delta^3 \to X$ such that:  
  \begin{itemize}
    \item The restriction of $T$ to the face missing $j_{\varphi}+1$ equals $\varphi$.
    \item The restriction of $T$ to the face missing $j_{\varphi}+2$ is a 2-simplex $\tau$ that  either factors through $X_{0}$ if $j_{\tau}=1$ or it has $j_{\tau} > j_{\varphi}$ if $j_{\tau}=0$.
    \item Every triangle of $T$ containing the edge $j_{\varphi}+1 \to j_{\varphi}+2$ is thin.
  \end{itemize}
  Note that by construction $f(T)$ must be fully scaled in $Y$. At this point the proof runs exactly the same as before. If the face of $T$ missing $j_{\varphi}+2$ is scaled it follows that $\varphi$ is scaled. This happens for example if $j_{\varphi}=1$. If $j_{\varphi}=0$ then we just need to apply \autoref{lem:inductionextravaganzza} twice to reach the conclusion. Finally if $s$ does not factor through $\Delta^1$ then it follows that $j_{\varphi}=1$ and the previous argument runs exactly the same. The proof for lean triangles is also essentially the same. If $s$ factors through $\Delta^n$ with $n\leq 1$ then the lean triangle in question must be thin and we are done. If this is not the case we reduce the problem to the previous situation with \autoref{lem:inductionextravaganzza}.
\end{proof}

\begin{definition}\label{defn:extravaganzadefn}
	Suppose we have a morphism
	\[
	\begin{tikzcd}
	X \arrow[rr,"f"] \arrow[rd,"p"] &   & Y \arrow[ld,swap,"q"] \\
	& \Delta^n &             
	\end{tikzcd}
	\]
	of 2-Cartesian fibrations over $\Delta^n$, for $n\geq 1$, and a commutative diagram 
	\[
	\begin{tikzcd}[ampersand replacement=\&]
	\partial \Delta^m \arrow[r,"\alpha"] \arrow[d] \& X \arrow[d,"f"] \\
	\Delta^m \arrow[r,"\beta"] \& Y
	\end{tikzcd}
	\]
	such that $r=q\circ \beta:\Delta^m \to \Delta^n$ is surjective.
	
	We define $j_\beta\in [m]$ to be the largest element such that $r(j_\beta)<r(m)$. We additionally define a simplicial subset $S^{m+1}_{\vec{j}}\subset \Delta^{m+1}$ to be the union of:
	\begin{itemize}
		\item all $m$-simplices of $\Delta^{m+1}$ \emph{other than} the faces missing $j_\beta+2$ or $j_\beta+1$;
		\item the $(m-1)$ simplex which misses \emph{both} $j_\beta+2$ and $j_\beta+1$.
	\end{itemize} 
 	See \autoref{fig:S31} for a geometric interpretation.
\end{definition}

\begin{figure}[h!]\label{fig:PathsInproduct}
	\begin{center}
		\begin{tikzpicture}[decoration={markings,mark=at position 0.5 with{\arrow{stealth}}}]
		\path (0,0) node (A0) {}
		(3,0) node (A2) {}
		(1.5,2.5) node (A1) {}
		(4,2) node (A3) {};
		\draw[blue,postaction={decorate}] (A0.center) -- (A1.center);
		\draw[blue,postaction={decorate}] (A0.center) -- (A3.center);
		\draw[blue,postaction={decorate}] (A0.center) -- (A2.center);
		\draw[white,line width=5mm] (A1) -- (A2);
		\path[fill=cyan,fill opacity=0.3] (A0.center) -- (A1.center) -- (A3.center)--cycle;
		\path[fill=cyan,fill opacity=0.3] (A0.center) -- (A1.center) -- (A2.center)--cycle;
		\draw[blue,postaction={decorate}] (A1.center) -- (A2.center);
		\draw[blue,postaction={decorate}] (A2.center) -- (A3.center);
		\draw[blue,postaction={decorate}] (A1.center) -- (A3.center);
		\end{tikzpicture}
	\end{center}
	\caption{The simplicial subset $S^3_1\subset S^3$.}\label{fig:S31}
\end{figure}
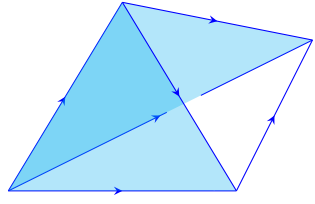

\begin{lemma}\label{lem:inductionextravaganzza}
	Let $n\geq 1$. Suppose we are given a morphism $f:X\to Y$ of 2-Cartesian fibrations over $\Delta^n$ and a lifting problem
	\[
	\begin{tikzcd}[ampersand replacement=\&]
	\partial \Delta^m \arrow[r,"\alpha"] \arrow[d] \& X \arrow[d,"f"] \\
	\Delta^m \arrow[r,"\beta"] \& Y
	\end{tikzcd}
	\]
	as in \autoref{defn:extravaganzadefn}. Suppose further that $f$ satisfies condition \ref{item:forExtravaganza} from \ref{prop:trivfibfibres}. 

  	 Then there exists a commutative diagram
  \[
    \begin{tikzcd}[ampersand replacement=\&]
      S^{m+1}_{\vec{j}} \arrow[r,"\epsilon"] \arrow[d] \& X \arrow[d,"f"] \\
      \Delta^{m+1} \arrow[r,"\theta"] \& Y
    \end{tikzcd}
  \]
   such that the following conditions hold:
  \begin{enumerate}
    \item The restriction of $\theta$ to be face missing $j_\beta+1$ equals $\beta$ and similarly, the restriction of $\epsilon$ to face missing $j_\beta +1$ equals $\alpha$.
    \item Let $\xi$ denote the restriction of $\alpha$ to the face missing $j_{\beta}+2$. Then either $j_{\xi} > j_{\beta}$ if $j_{\beta}<m-1$ or $\xi$ factors through $\Delta^{n-1}$ and similarly for $\epsilon$.
    \item The edge $j_{\beta}+1 \to j_{\beta}+2$ is marked and every triangle of $\theta$ that contains it is lean in $Y$. The analogous statement holds for $\epsilon$.
  \end{enumerate}
\end{lemma}
\begin{proof}
  We start the proof by fixing the notation $\alpha(i)=x_i$ (resp. $\beta(i)=y_i$). Let us pick a marked morphism $e:\hat{x}_{j_\beta} \to x_{j_{\beta}+1}$. To ease notation, let us just denote $j_{\beta}$ simply by $j$. We define mb simplicial sets
  \[
    B^m_j=\Delta^m \coprod\limits_{\Delta^{\{j+1\}}}(\Delta^1)^{\sharp} \quad , \quad \partial B^m_j= \partial \Delta^m \coprod\limits_{\Delta^{\{j+1\}}} (\Delta^1)^{\sharp}.
  \]
  Note that we have commutative diagrams
  \[
    \begin{tikzcd}[ampersand replacement=\&]
      B^m_j \arrow[r] \arrow[d,swap,"\gamma^m_j"] \& Y \arrow[d,"q"] \\
      \Delta^{m+1} \arrow[r,"r \circ s_j"] \& \Delta^n
    \end{tikzcd} \quad \quad \quad
     \begin{tikzcd}[ampersand replacement=\&]
      \partial B^m_j \arrow[r] \arrow[d,swap,"\iota^m_j"] \& X \arrow[d,"p"] \\
      S^{m+1}_{\vec{j}} \arrow[r] \& \Delta^n
    \end{tikzcd}
  \]
 where bottom horizontal map in the second diagram is the restriction of $r \circ s_j$ to $S^{m+1}_{\vec{j}}$. We claim that the left vertical maps in both diagrams are \bS-anodyne. Suppose that this was already proved. Then let $\theta$ be a solution to the left-most commutative square. Note that we can form another diagram
 \[
   \begin{tikzcd}[ampersand replacement=\&]
      \partial B^m_j \arrow[r] \arrow[d] \& X \arrow[d,"f"] \\
      S^{m+1}_{\vec{j}} \arrow[r] \& Y
    \end{tikzcd} 
 \]
 where bottom horizontal map is the composite $\func{S^{m+1}_{\vec{j}} \to \Delta^{m+1} \to[\theta] Y}$. Since $f$ has the right lifting property against \bS-anodyne morphisms our result follows.

 First we will prove the family of cases where $j=m-1$ by using induction on $m$. The case $m=1$ is obviously true. Suppose that our claim holds for $m-1$ and let us prove the case $m$. We define simplicial sets $Y_i$ (resp. $\partial Y_i$) inductively by attaching to $Y_{i-1}$ (resp. $\partial Y_{i-1}$) the face missing $i$ for $0\leq i \leq m-1=j$. We make the convention $Y_{-1}=B^m_{m-1}$ (resp. $\partial Y_{-1}=\partial B^m_{m-1}$). This yields a filtrations
  \[
    \func{Y^m_{-1} \to Y^m_0 \to \cdots \to  Y^m_{m-1}=\Lambda^{m+1}_{m+1}} \quad, \quad \func{\partial Y^m_{-1} \to \partial Y^m_0 \to \cdots \to \partial Y^m_{m-1}=S^{m}_{\vec{j}}}
  \]
  It will then suffice to show that each step in both filtrations is \bS-anodyne. Let $0\leq i \leq m-1$ then we can produce a pushout squares
  \[
   \begin{tikzcd}[ampersand replacement=\&]
     Y^{m-1}_{i-1} \arrow[d] \arrow[r] \& \Delta^{n-1} \arrow[d] \\
     Y^{m}_{i-1} \arrow[r] \& Y^{n}_{i}
   \end{tikzcd} \quad \quad 
   \begin{tikzcd}[ampersand replacement=\&]
     Y^{m-1}_{i-1} \arrow[d] \arrow[r] \& \Delta^{n-1} \arrow[d] \\
     \partial Y^{m}_{i-1} \arrow[r] \& \partial Y^{n}_{i}
   \end{tikzcd}
  \]
  and the claim holds from the inductive hypothesis. The general proof will employ induction on $j$ and each case will be proved using induction on $m$. Note that given $j \geq 0$ the ground case for the induction on $m$ is given by $m=j+1$. In particular we have proved all the ground cases already. Now we will deal with ground case of the induction on $j$, namely $j=0$. Assume the claim to hold for $m-1 \geq 1$ and let us prove the case $m$. We define inductively $Z^m_{i-1}$ by attaching to $Z^m_{i}$ the face missing $i-1$ for $3\leq i\leq m+2$ where we define $Z^m_{m+2}=Z^m_0$ and analogously for $\partial Z^m_i$. This yields filtrations
  \[
    \func{Z^m_{m+2} \to Z^m_{m+1} \to \cdots \to  Z^m_{3}\to \Lambda^{m+1}_{2} } 
  \]
  \[
    \func{\partial Z^m_{m+2} \to \partial Z^m_{m+1} \to \cdots \to \partial Z_{3} \to S^{m}_{\vec{j}}}
  \]
  where the last step in both filtrations is given by attaching the face missing $0$. A similar argument as above shows that the claim follows from the inductive hypothesis for the every step except the last one. To prove that the last map in both filtrations is \bS-anodyne we consider pushout diagram
  \[
    \begin{tikzcd}[ampersand replacement=\&]
      \Lambda^{m-1}_1 \arrow[r] \arrow[d] \& \Delta^{m-1} \arrow[d] \\
      Z_{3} \arrow[r] \& \Lambda^{m+1}_{2}
    \end{tikzcd}
  \]
  and note that the triangle $\set{0,1,2}$ must be already be thin if $m>3$ or it can be chosen to be thin since it lies above a degenerate triangle in $\Delta^n$. The analogous conclusion also holds for $\partial Z_3$. Finally let us assume the claim holds for $j-1\geq 0$. The proof of this final inductive hypothesis is a mix of both previous cases. We will give an sketch here and leave the details for the interested reader. The idea is to add stepwise to $B^n_j$ (resp. $\partial B^n_j$) the faces missing $i$ for $n \leq i \leq j+3$. One can check that at each step this result map is \bS-anodyne using the induction hypothesis. Then we add the faces missing $\ell$ for $0 \leq \ell \leq j$ and again we find that each step in this process is \bS-anodyne. In the case of $\partial B^m_j$ we have already reached $S^{m+1}_{\vec{j}}$.  For $B^m_j$ after this process we reach $\Lambda^{m+1}_{j+2}$ where the triangle $\set{j+1,j+2,j+3}$ must be thin and the conclusion follows.
\end{proof}

\begin{proposition}\label{prop:fibresdetectequivalences}
 Suppose we are given a morphism of 2-Cartesian fibrations
  \[
    \begin{tikzcd}
     X \arrow[rr,"f"] \arrow[rd,"p"] &   & Y \arrow[ld,swap,"q"] \\
                        & S &             
    \end{tikzcd}
  \]
  Then the following are equivalent
  \begin{itemize}
    \item[i)] The map $f$ is a weak equivalence.
    \item[ii)] For every $s \in S$ the induced morphism $f_s:X_s \to Y_s$ is an equivalence of scaled simplicial sets.
  \end{itemize}
\end{proposition}
\begin{proof}
  The implication $i) \implies ii)$ is clear since we can construct an inverse up to homotopy for $f$ as we did in the proof of \autoref{prop:weakequivkan}. To prove the converse we will apply the small object argument and obtain a factorization of $f$
  \[
    \func{X \to[u] L \to[v] Y}
  \]
  where the map $u$ is \bS-anodyne and $v$ has the right lifting property against the class of \bS-anodyne maps. It follows from \autoref{prop:mappingcofbs} that $u$ must be a weak equivalence. Now we observe that $L \to S$ must be a 2-Cartesian fibration. It follows that the induced morphism on fibres $L_s \to Y_s$ must be a bicategorical equivalence for every $s \in S$. We can now apply \autoref{prop:trivfibfibres} to obtain that $v$ must be a trivial fibration. This finishes the proof.
\end{proof}

\begin{lemma}\label{lem:perfect}
  The class of weak equivalences in $(\mbsSet)_{/S}$ is perfect in the sense of \cite[A.2.6.10]{HTT}.
\end{lemma}
\begin{proof}
  Using the small object argument we produce a functor (which preserves filtered colimits)
  \[
    \func{T:(\mbsSet)_{/S} \to (\mbsSet)_{/S}}
  \]
  equipped with a natural transformation $\on{id} \xRightarrow T$ such that for every $K \in (\mbsSet)_{/S}$ the map $K \to T(K)$ is \bS-anodyne and $T(K)$ is a 2-Cartesian fibration. It follows that a morphism $h:K \to L$ is a weak equivalence if and only if $T(h)$ is a weak equivalence. We finally consider the composite
  \[
    \func{W_S:(\mbsSet)_{/S}  \to[T] (\mbsSet)_{/S}  \to \prod\limits_{s \in S}\mbsSet \to \prod\limits_{s \in S}\on{Set}_{\Delta}^{\on{sc}} }
  \]
  where the second functor is given by taking pullback along each fibre and the second functor is a product of forgetful functors. It follows that $W_S$ preserves filtered colimits. Let $\mathcal{E}_S=\prod\limits_{s \in S}\mathcal{E}$ where $\mathcal{E}$ denotes the collection of weak equivalences in $\on{Set}_{\Delta}^{\on{sc}}$. Since $\mathcal{E}$ is perfect then so is $\mathcal{E}_S$. It is immediate to see that the collection of weak equivalences in $(\mbsSet)_{/S}$ is precisely given by $W_{S}^{-1}(\mathcal{E}_S)$ and the result follows from \cite[A.2.6.12]{HTT}.
  \end{proof}

  \begin{lemma}\label{lem:compatible}
    Let $p:X \to S$ and $n\geq 0$. Then the morphism $r:X \times (\Delta^n)_{\sharp}^{\sharp} \to X$ given by projection to $X$ is a weak equivalence.
  \end{lemma}
  \begin{proof}
    Note that the inclusion of the terminal object $t_n:(\Delta^{0})^{\sharp}_{\sharp} \to (\Delta^n)^{\sharp}_{\sharp}$ induces a section $s:X \to X \times (\Delta^n)^{\sharp}_{\sharp}$. Since our class of weak equivalences satisfies 2-out-of-3 it follows that it is enough to show that $s$ is a weak equivalence. It is easy to verify that the map $t_n$ is \bS-anodyne and the claim follows from \autoref{prop:PP}.
  \end{proof}
  
  \begin{theorem}\label{thm:modelstructure}
    Let $S$ be a scaled simplicial set. Then there exists a left proper combinatorial simplicial model structure on $(\mbsSet )_{/S}$, which is characterized uniquely by the following properties:
    \begin{itemize}
      \item[C)] A morphism $f:X \to Y$ in $(\mbsSet )_{/S}$ is a cofibration if and only if $f$ induces a monomorphism betwee the underlying simplicial sets.
      \item[F)] An object $X \in (\mbsSet )_{/S}$ is fibrant if and only if $X$ is a 2-Cartesian fibration. 
    \end{itemize}
  \end{theorem}
  \begin{proof}
    We will use \cite[Prop. A.2.6.13]{HTT} to deduce the existence of a left proper combinatorial model structure in $(\mbsSet )_{/S}$. \autoref{lem:perfect} shows that the class of weak equivalences is perfect. We proved in \autoref{prop:leftproper} that weak equivalences are stable under pushouts along cofibrations. It is also immediate to see that trivial fibrations are in particular weak equivalences so the conditions of \cite[Prop. A.2.6.13]{HTT} apply. Now we wish to show that this model structure is compatible with the simplicial structure. This follows from \cite[Prop. A.3.1.7]{HTT} coupled with \autoref{lem:compatible}.
  \end{proof}

  \begin{theorem}\label{thm:equivwithSC}
    The adjunction presented in \autoref{rem:forgetfuladj}
    \[
      \begin{tikzcd}[ampersand replacement=\&]
        L:\on{Set}_{\Delta}^{\on{sc}} \arrow[r,shift left=0.8] \& \arrow[l,shift left=0.8] \mbsSet:U
      \end{tikzcd}
    \]
    is a Quillen equivalence where the right-handside is equipped with the model structure of mb simplicial sets over the point constructed in \autoref{thm:modelstructure}.
  \end{theorem}
  \begin{proof}
    First we will show that $L$ preserves cofibrations and trivial cofibrations. The case of cofibrations is immediate. Now let us suppose that $(A,T_A) \to (B,T_B)$ is a trivial cofibration of scaled simplicial sets. Let $\mathbb{D}$ be a fibrant object in $\mbsSet$ and note that as stated before $\mathbb{D}$ is an $\infty$-bicategory with all the equivalences marked. It is immediate that the morphism
    \[
      \func{\on{Fun}^{\mathbf{mb}}(L(B),\mathbb{D}) \to  \on{Fun}^{\mathbf{mb}}(L(A),\mathbb{D}))}
    \]
    can be identified with the analogous morphism
    \[
      \func{\on{Fun}^{\mathbf{sc}}(B,U(\mathbb{D})) \to  \on{Fun}^{\mathbf{sc}}(A,U(\mathbb{D})}
    \]
    between the underlying scaled simplicial sets. It follows that $L \dashv U$ is a Quillen adjunction. Note that $U \circ L=\on{id}$. To conclude the proof suppose that $\mathbb{B}$ is a fibrant mb simiplicial set. In particular, we need to show that the map
    \[
     \func{(\mathbb{B},\flat,T_{\mathbb{B}}) \to (\mathbb{B},E_{\mathbb{B}},T_{\mathbb{B}}) }
    \]
    is a weak equivalence. However the above morphism is a pushout of a morphism of type \ref{mb:equivalences} in \autoref{def:mbsanodyne}.
    \end{proof}

\section{2-Cartesian fibrations over a fibrant base.}\label{sec:fibbase}
The goal of this section is to give a characterization of 2-Cartesian fibrations in the specific case where $S \in \on{Set}_{\Delta}^{\on{sc}}$ is an $\infty$-bicategory. For the rest of this section we will fix a functor of $\infty$-bicategories $p: X\to S$.

\begin{definition}
  We say that a 2-simplex $\sigma: \Delta^2 \to X$ is \emph{left degenerate} if its restriction $\restr{\sigma}{\Delta^{\set{0,1}}}$ is a degenerate simplex in $X$.
\end{definition}

\begin{definition}
  Let $\func{p:X \to S}$ be a weak \sS-fibration. We call a left-degenerate 2-simplex $\sigma: \Delta^2  \to X$, $p$-\emph{coCartesian} if there exists a solution for any lifting problem of the form
  \[
    \begin{tikzcd}[ampersand replacement=\&]
      \Lambda^n_0 \arrow[r,"f"] \arrow[d] \& X \arrow[d,"p"] \\
      \Delta^{n} \arrow[r] \& S
    \end{tikzcd}
  \]
  provided $\restr{f}{\Delta^{\set{0,1,n}}}=\sigma$.
\end{definition}

\begin{remark}
  It is immediate to see that a coCartesian 2-simplex $\sigma:\Delta^2 \to X$ defines in particular, a coCartesian edge in the mapping space $X(a,b)$ described in \autoref{defn:mappingspace} where $\sigma(0)=a$ and $\sigma(2)=b$. We wish to show that this property precisely characterizes coCartesian triangles. The proof of this later fact will involve a little bit of work.
\end{remark}

\begin{lemma}\label{lem:n+1face}
  Let $X$ be an $\infty$-bicategory and consider an $n$-simplex $\sigma: \Delta^n \to X$. Suppose that there exists some $k<n$ such that the restriction of $\sigma$ to $\Delta^{[0,k]}$ is degenerate on $\sigma(0)=a$. Then there exists a morphism
  \[
    \func{\hat{\sigma}:\Delta^{n+1} \to X}
  \]
with the following properties:
\begin{itemize}
  \item The restriction of $\hat{\sigma}$ to its $(k+1)$-face equals $\sigma$.
  \item The restriction of $\hat{\sigma}$ to $\Delta^{[0,k+1]}$ is degenerate on $a$.
  \item For every $k+2 \leq j \leq n+1$ the 2-simplex $\Delta^{\set{k+1,k+2,j}}$ is thin in $X$.
\end{itemize}
\end{lemma}
\begin{proof}
  Our first observation is that if $k=n-1$ then we can define $\hat{\sigma}=s_{n-1}(\sigma)$ and this provides the desired solution. We will assume for the rest of the proof that $n-k >1$. In order to tackle the more general cases we define a subsimplicial set $\iota:R^{n}_k \to \Delta^{n+1}$ consisting precisely of those simplices $\theta:\Delta^k \to \Delta^{n+1}$ satisfying \emph{at least one} of the following conditions
  \begin{itemize}
    \item[a)] The simplex $\theta$ skips the vertex $k+1$.
    \item[b)] The simplex $\theta$ skips the vertex $n+1$.
    \item[c)] The simplex $\theta$ is one of the triangles $\Delta^{\set{k+1,k+2,j}}$ for $k+2 <j \leq n+1$
  \end{itemize}
  We endow $\Delta^{n+1}$ with a scaling by scaling those triangles contained in $\Delta^{[0,k+1]}$ in addition to the triangles $\Delta^{\set{k+1,k+2,j}}$ for $k+2 <j \leq n+1$. The proof will performed in two steps: First we will show that $\iota$ is an scaled anodyne morphism. Finally, we will produce an extension of $\sigma$ to $R^n_k$. 

  Inductively define simplicial sets $A^n_{(k,i)}$ from $A^n_{(k,i-1)}$ by adding the face missing $i$ where $1 \leq i \leq k$ and where we are using the convention $A^{n}_{(k,0)}=R^n_k$. Let $B^n_{(k,n)}$ be the simplicial set obtained from $A^n_{(k,k)}$ by adding the face missing $n$. We inductively define $B^n_{(k,j)}$ from $B^n_{(k,j+1)}$ by adding the face missing $j$ for $k+3\leq j \leq n$ with the convention $A^n_{(k,k)}=B^{n}_{(k,n+1)}$. This yields a filtration
  \[
    \func{R^n_k \to A^{n}_{(k,1)} \to \cdots \to A^n_{(k,k)}\to B^n_{(k,n)} \to  \cdots \to B^n_{(k,k+3)} \to \Delta^{n+1}}
  \]
  We wish to show that each step in the filtration is given by an scaled anodyne morphism. Note that $B^n_{(k,k+3)}$ contains all faces except the face missing $0$ and the face missing $k+2$. Since the triangle $\Delta^{\set{k+1,k+2,k+3}}$ is thin it is easy to verify that the last step in our filtration is scaled anodyne. We observe that we can produce pushout diagrams
  \[
    \begin{tikzcd}[ampersand replacement=\&]
      A^{n-1}_{(k-1,i)} \arrow[r] \arrow[d] \& \Delta^n \arrow[d] \\
      A^n_{(k,i)} \arrow[r] \& A^n_{(k,i+1)}
    \end{tikzcd} \quad \quad \begin{tikzcd}[ampersand replacement=\&]
      B^{n-1}_{(k,j-1)} \arrow[r] \arrow[d] \& \Delta^n \arrow[d] \\
      B^n_{(k,j)} \arrow[r] \& B^n_{(k,j-1)}
    \end{tikzcd}
  \]
which hints at the existence of an inductive proof. The ground case we need to show is $n=3$ and $k=2$. In this setting the filtration is of the form
\[
  \func{R^3_2 \to A^{3}_{(1,1)} \to \Delta^{4}.}
\]
It is an easy exercise to verify the claim in this particular case. The claim follows easily by induction.

To finish the proof we need to produce the extension from $\sigma: \Delta^n \to X$ to a map $\rho:R^n_k \to X$. We define $L^{n}_k$ as the subsimplicial of $R^n_k$ consisting in those simplices satisfying conditions $a)$ or $c)$. We define $\rho(k+1 \to k+2)=\sigma(k \to k+1)$ and extend $\sigma$ to $L^n_k$ by picking the obvious composites of morphisms. Note that if $n-k=2$ then we can produce the desired extension by just setting $d_{n+1}(\rho)=s_{k}(d_n(\sigma))$. Therefore will assume that $L^n_k$ already contains those simplices that factor through $\Delta^{[0,k+2]}$. To finish the proof we will show that $L^n_k \to R^n_k$ is scaled anodyne. 
\[
  \alpha_{k+j}: \Delta^{[0,k+1]} \to R^n_k \quad, \quad \text{ for } 3\leq j \leq n-1
\]
Let us set $C^n_{(k,2)}=L^n_k$. We define inductively $C^n_{(k,j)}$ by attaching the simplices $\alpha_{k+j}$ to $C^n_{(k,j-1)}$. We obtain our final filtration
\[
  \func{L^n_k \to C^n_{(k,3)} \to \cdots \to C^n_{(k,n-1)}=R^n_k}.
\]
Note that we have pushout diagrams
\[
  \begin{tikzcd}[ampersand replacement=\&]
    R^{[0,k+j]}_{k} \arrow[r] \arrow[d] \& \Delta^{[0,k+j]} \arrow[d,"\alpha_k"] \\
    C^{n}_{(k,j)} \arrow[r] \& C^{n}_{(k,j+1)}
  \end{tikzcd}
\]
where the top horizontal morphism is scaled anodyne by the first part of this proof. The result follows.
\end{proof}

\begin{proposition}\label{prop:characterizationcoCart}
  Let $p: X \to S$ be a weak \sS-fibration. Then a left-degenerate triangle $\sigma: \Delta^2 \to X$  with $\sigma(0)=a$ and $\sigma(2)=b$ is coCartesian if and only if it defines as coCartesian edge in the mapping space $X(a,b)$.
\end{proposition}
\begin{proof}
  It is immediate that if $\sigma$ is coCartesian then it defines a coCartesian edge in the corresponding mapping space. For the converse let $n\geq 3$ and consider a lifting problem 
  \[
    \begin{tikzcd}[ampersand replacement=\&]
      \Lambda^n_0 \arrow[r,"f"] \arrow[d] \& X \arrow[d,"p"] \\
      \Delta^{n} \arrow[r,"\alpha"] \& S
    \end{tikzcd}
  \]
  such that $\restr{f}{\Delta^{\set{0,1,n}}}=\sigma$. We define $1 \leq k \leq n-1$ to be the biggest integer such that the restriction of $f$ to $\Delta^{[0,k]}$ is degenerate on $a$. Note that if $n-k=1$ then the lifting problem takes place in the mapping space $X(a,b)$ and the solution is guaranteed. We define a subsimplicial set $P^{n}_k \subset \Delta^{n+1}$ consisting of those simplices $\rho: \Delta^k \to \Delta^{n+1}$ satisfying \emph{at least one} of the following conditions
  \begin{itemize}
    \item[a)] The simplex $\rho$ skips the vertex $n+1$.
    \item[b)] The simplex $\rho$ skips a pair of vertices $(k+1,i)$ with $i\neq 0$.
    \item[c)] The simplex $\rho$ factors through $\Delta^{\{k+1,k+2,j\}}$ with $k+2<j \leq n+1$.
  \end{itemize}
 Now we can apply \autoref{lem:n+1face} to the simplex $\alpha$ to obtain a map $\hat{\alpha}:\Delta^{n+1}\to S$ satisfying the conditions stated in the lemma. Our first goal is to produce a commutative diagram
  \[
    \begin{tikzcd}[ampersand replacement=\&]
      P^{n}_k \arrow[r,"\hat{f}"] \arrow[d] \& X \arrow[d,"p"] \\
      \Delta^{n+1} \arrow[r,"\hat{\alpha}"] \arrow[ur,dotted,"\epsilon"] \& S
    \end{tikzcd}
  \]
  since any dotted arrow as above will provide a solution to the original lifting problem. Let $M^n_k \subset P^n_k$ be the subsimplicial subset consisting in those simplices satisfying $a)$ or $c)$. It follows easily from the assumption that $p$ is a weak \sS-fibration that we have a commutative diagram
  \[
    \begin{tikzcd}[ampersand replacement=\&]
      M^{n}_k \arrow[r] \arrow[d] \& X \arrow[d,"p"] \\
      \Delta^{n+1} \arrow[r,"\hat{\alpha}"]  \& S
    \end{tikzcd}
  \]
  It follows from \autoref{lem:n+1face} that the restriction of $\hat{\alpha}$ to $\Delta^{[0,k+2]}$ equals $s_{k}(d_n(\alpha))$. In particular  we can assume that $M^n_k$ contains the simplex $\Delta^{[0,k+2]}$. We then observe that we have a pushout diagram
  \[
    \begin{tikzcd}[ampersand replacement=\&]
      L^{n-1}_k \arrow[r] \arrow[d] \& \Delta^n \arrow[d] \\
      M^n_k \arrow[r] \& P^n_k
    \end{tikzcd}
  \]
where the top horizontal morphism is scaled anodyne. It follows that the desired extension $\hat{f}:P^n_k \to X$ exists. To finish the proof we will construct the dotted arrow $\epsilon$ above. Let $S^n_k$ be the subsimplicial subset of $\Delta^{n+1}$ consisting in those simplices belonging to $P^n_k$ in addition to the faces  that skip the vertices $i$ for $1 \leq i \leq k$. A totally analogous argument as that for \autoref{lem:n+1face} shows that the map inclusion $P^n_k \to S^n_k$ is scaled anodyne. We can now add the faces that skip the vertices $k+2 \leq j\leq n$ to obtain a new simplicial set $T^n_k$. We observe that $T^n_k$ only misses the $(k+1)$-face and the $0$-face since the triangle $\Delta^{\set{k,k+1,k+2}}$ must be thin. It is easy to see that $T^n_k \to \Delta^{n+1}$ is scaled anodyne. To finish the proof provide a solution to the lifting problem
  \[
    \begin{tikzcd}[ampersand replacement=\&]
      S^{n}_k \arrow[r] \arrow[d] \& X \arrow[d,"p"] \\
      T^n_k \arrow[r] \arrow[ur,dotted,"\varphi"] \& S
    \end{tikzcd}
  \]
  We define $D^{n}_{(k,n)}$ by adding to $S^n_k$ the face missing the vertex $n$. We define $D^{n}_{(k,j-1)}$ by adding to $D^n_{(k,j)}$ the face missing $j$ for $k+2 \leq j \leq n$. This produces a filtration
  \[
    \func{S^n_k \to D^{n}_{(k,n)} \to \cdots D^{n}_{(k,k+2)}=T^n_k}
  \]
We will show how to produce the solution by extending the map stepwise. As usual, we produce a pushout diagram
\[
  \begin{tikzcd}[ampersand replacement=\&]
    D^{n-1}_{(k,j-1)} \arrow[r] \arrow[d] \& \Delta^n \arrow[d] \\
    D^n_{(k,j)} \arrow[r] \& D^n_{(k,j-1)}
  \end{tikzcd}
\]
Now we observe that if $n-k=2$ then original filtration is of the form
\[
  \func{S^n_{n-2} \to D^n_{(n-2,n)}=T^n_{n-2}}
\]
the previously depicted pushout diagram particularizes now to
\[
  \begin{tikzcd}[ampersand replacement=\&]
    \Lambda^{n}_0 \arrow[r] \arrow[d] \& \Delta^n \arrow[d] \\
    S^n_{n-2} \arrow[r] \& T^n_{n-2}
  \end{tikzcd}
\]
where the left-most $\Lambda^n_0$ represents an $0$-horn in the mapping space and thus the existence of the extension is guaranteed. An inductive argument shows that we can produce the map $\varphi$ and the proof is concluded.
\end{proof}

\begin{definition}
  We say that $p:X \to S$ is \emph{locally 
  fibred} if it satisfies the conditions
  \begin{itemize}
     \item[i)] The map $p: X \to S$ is a weak \sS-fibration.
     \item[ii)] For every left-degenerate $\widetilde{\sigma}: \Delta^2 \to S$ together with  $\tau:\Delta^1 \to X$ such that $\restr{\widetilde{\sigma} }{\Delta^{\set{0,2}}}=p(\tau)$, then there exists a left-degenerate simplex $\sigma: \Delta^2 \to X$ such that $\sigma$ is coCartesian and $p(\sigma)=\widetilde{\sigma}$.
  \end{itemize} 
\end{definition}

The following proposition follows immediately from our definitions.

\begin{proposition}\label{prop:mappingcoCart}
 Let $p:X \to S$ locally 
 fibred. The given $a,b \in X$ a pair of objects it follows that the induced morphism on mapping spaces
 \[
   \func{p_{a,b}: X(a,b)\to S(p(a),p(b))}
 \]
 is a coCartesian fibration of $\infty$-categories.
\end{proposition}

\begin{definition}
  Let $\sigma,\tau:\Delta^2 \to X$ be a pair of 2-simplices such that $\tau$ is left-degenerate. We say $\tau$ is the \emph{left-degeneration} of $\sigma$ if there exists a 3-simplex $\rho:\Delta^3 \to X$ with the following properties:
  \begin{itemize}
    \item The face $d_3(\rho)$ equals $s_0(d_2(\sigma))$.
    \item The face $d_2(\rho)$ equals $\tau$.
    \item The face $d_1(\rho)$ equals $\sigma$.
    \item The face $d_0(\rho)$ is thin in $X$.
  \end{itemize}
\end{definition}

\begin{remark}
  We remark that if $X$ is an $\infty$-bicategory then the left-degeneration of a 2-simplex always exist. It is trivial to see that every left-degenerate triangle is its own left-degeneration.
\end{remark}

\begin{definition}
  We say that a triangle $\sigma:\Delta^2 \to X$ is \emph{coCartesian} if its left-degeneration is coCartesian. We denote the collection of coCartesian triangles by $C_X$.
\end{definition}

\begin{definition}
  Let $p:X \to S$ be a weak \sS-fibration.  We say that the collection of coCartesian triangles $C_X$, is a \emph{functorial family} if the following holds:
  \begin{itemize}
    \item Let $0<i<3$ and suppose we are given a three simplex $\rho: \Delta^3 \to X$ such that the face $\Delta^{\set{i-1,i,i+1}}$ is thin and all of the faces of $\rho$ are coCartesian except possibly the face missing $i$. Then the image of $\rho$ only consists in coCartesian triangles.
  \end{itemize}
\end{definition}

\begin{definition}
  Let $p:X \to S$ be a locally fibred morphism. We say that $p$ is \emph{functorially fibred} if the collection of coCartesian triangles is functorial.
\end{definition}

\begin{lemma}\label{lem:fibrants4}
  Let $p: X \to S$ be a functorially fibred map. Given a left-degenerate three simplex $\rho: \Delta^3 \to X$ such that all of its faces except possibly the 0-face belong to $C_X$ then it follows that the 0-face must also belong to $C_X$. 
\end{lemma}
\begin{proof}
  This lemma is an immediate application of \autoref{lem:n+1face}.
\end{proof}

\begin{definition}
  Let $p:X \to S$ be a functorially fibred map. We say that an edge $e:\Delta^1 \to X$ is \emph{strongly} p-cartesian (resp. $p$-Cartesian) if every lifting problem
  \[
     \begin{tikzcd}[ampersand replacement=\&]
    \Lambda^n_n \arrow[r,"f"] \arrow[d] \& X \arrow[d,"p"] \\
    \Delta^{n} \arrow[r] \arrow[ur,"\hat{f}"] \& S
     \end{tikzcd}
  \]
  admits a solution for $n\geq 2$ provided the following conditions are satisfied:
  \begin{itemize}
    \item[i)] $\restr{f}{\Delta^{\set{n-1,n}}}=e$.
    \item[ii)] $\restr{f}{\Delta^{\set{0,n-1,n}}}$ is coCartesian (resp thin).
  \end{itemize}
  In the case $n=2$, we  can always produce a coCartesian triangle (resp. thin) $\hat{f}$ as long as $\restr{f}{\Delta^{\set{n-1,n}}}=e$.
\end{definition}

\begin{definition}
  We say that a functorially fibred map $p:X \to S$ is an outer 2-fibration or $\mathbf{O2}$-fibration, if every degenerate edge is $p$-Cartesian.
\end{definition}

\begin{remark}
  Observe that an $\mathbf{O2}$-fibration is a \emph{weak fibration} in the terminology of \cite{GHL_LaxLim}.
\end{remark}

\begin{lemma}\label{lem:fibrants5}
  Let $p:X \to S$ be a $\mathbf{O2}$-fibration. Given a 3-simplex $\rho:\Delta^3 \to X$ such that
  \begin{itemize}
    \item The restriction $\restr{\rho}{\Delta^{\set{2,3}}}$ is a $p$-Cartesian edge.
    \item Every face of $\rho$ belongs to $C_X$ except possibly the face missing 3.
  \end{itemize}
  Then every face of $\rho$ belongs to $C_X$.
\end{lemma}
\begin{proof}
  Let us fix some notation before diving into the proof. We denote $a=\rho(0)$, $c=\rho(2)$, $d=\rho(3)$ and $\restr{\rho}{\Delta^{\{2,3\}}}=\alpha$. First let us assume that $\restr{\rho}{\Delta^{\set{0,1}}}$ is degenerate on $a$. Using Proposition 2.3.3 in \cite{GHL_LaxLim} we obtain a homotopy pullback diagram
  \[
    \begin{tikzcd}[ampersand replacement=\&]
      X(a,c) \arrow[r,"\alpha \circ \mathblank"] \arrow[d] \& X(a,d) \arrow[d] \\
      S(p(a),p(c)) \arrow[r,"p(\alpha) \circ \mathblank"] \& S(p(a),p(d))
    \end{tikzcd}
  \]
  Let $\epsilon \in X(a,c)$ denote the morphism represented by the $3$-face in $\rho$. Then it follows from our hypothesis that its image under postcomposition with $\alpha$ must be coCartesian. Since $X(a,c)$ can be expressed as a homotopy pullback it follows that $\epsilon$ must be coCartesian as an edge in $X(a,c)$. The claim now follows from \autoref{prop:characterizationcoCart}

  To prove the general version of the lemma we will reduce it to the previous case. We will fix once and for all the notation regarding $\rho$ by means of the diagram below
  
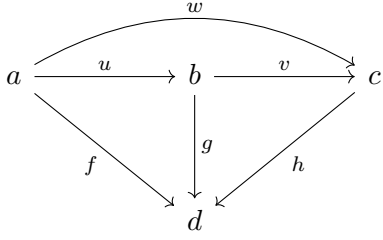
\[
    \begin{tikzcd}
     a \arrow[rr, "u"] \arrow[rrrr, "w", bend left] \arrow[rrdd,swap, "f"] &  & b \arrow[rr, "v"] \arrow[dd, "g"] &  & c \arrow[lldd, "h"] \\
     &  &                                   &  &                     \\
     &  & d                                 &  &                    
    \end{tikzcd}
  \]
  Let us consider $\Lambda^2_1$ sitting inside the 3-face of $\rho$ and another such horn sitting inside the 2-face of $\rho$. Let $\sigma_3$ (resp. $\sigma_2$) denote the corresponding thin 2-simplices obtained by extending the horns. We denote the 1-face of these thin simplices by $v \circ u$ (resp. $g \circ u$). We define a morphism
  \[
    \func{\Lambda^3_2 \to X}
  \]
  by sending the $0$-face to $\sigma_3$, the $1$-face to $d_3(\rho)$, the $3$-face to $s_0(u)$. Since $X$ is an $\infty$-bicategory we can produce a lift to a 3-simplex that we call $\theta_4$. We note that $d_3(\rho)$ is coCartesian if and only if $d_2(\theta_4)$ belongs to $C_X$. We define a morphism
  \[
    \func{\Lambda^3_1 \to X}
  \]
by sending the $0$-face to $d_0(\rho)$, the $2$-face to $\sigma_2$ and the $3$-face to $\sigma_3$. We extend this horn to a $3$-simplex that we call $\theta_0$. It follows that every face of $\theta_0$ belongs to $C_X$. Finally, let us define
  \[
    \func{\Lambda^3_2 \to X}
  \]
  by sending the $0$-face to $\sigma_2$, the $1$-face to $d_2(\rho)$ and the $3$-face to $s_0(u)$. We call $\theta_3$ the extension of this horn to a $3$-simplex. We observe that every face of $\theta_3$ belongs to $C_X$.

  Let $\theta_1=\rho$ and observe that the 3-simplices $\theta_i$ for $i\in [4]$, $i \neq 2$ assemble into a $\Lambda^4_2$ and that the face $\Delta^{\set{1,2,3}}$ is thin by construction. We take our final extension $\theta: \Delta^4 \to X$ and observe that $d_2(\theta)$ satisfies the conditions of the lemma and its first edge is degenerate. We finish the proof by noting that $d_3 d_2(\theta)$ is coCartesian if and only if $d_3(\rho)$ is. 
\end{proof}

\begin{proposition}\label{prop:quasi}
  Let $p:X \to S$ be an $\mathbf{O2}$-fibration. Then an edge $e:\Delta^1 \to X$ is  strongly $p$-Cartesian if and only if it is $p$-Cartesian.
\end{proposition}
\begin{proof}
  The `only if' direction is obvious. Now let us suppose that $e$ is $p$-Cartesian and consider a lifting problem
  \[
     \begin{tikzcd}[ampersand replacement=\&]
    \Lambda^n_n \arrow[r,"f"] \arrow[d] \& X \arrow[d,"p"] \\
    \Delta^{n} \arrow[r,"\sigma"] \arrow[ur,dotted,"\hat{f}"] \& S
     \end{tikzcd}
  \]
  such that $\restr{f}{\Delta^{\set{n-1,n}}}=e$ and $\restr{f}{\Delta^{\set{0,n-1,n}}}$ belongs to $C_X$. We define
  \[
    \func{M_n=\Lambda^n_n \coprod\limits_{\Delta^{\{n\}} }\Delta^1 \to[t] X }
  \]
  where $t$ maps $\Lambda^n_n$ via $f$ and $\Delta^1$ via $e$. We have a canonical inclusion $\iota:M_n \to \Delta^{n+1}$ where $\Lambda^n_n$ lands in the face missing $n$ and $\Delta^1$ gets identified with the final edge. We produce a commutative diagram
   \[
     \begin{tikzcd}[ampersand replacement=\&]
    M^n \arrow[r,"t"] \arrow[d] \& X \arrow[d,"p"] \\
    \Delta^{n+1} \arrow[r,"\ell"] \arrow[ur,dotted,"\hat{t}"]  \& S
     \end{tikzcd}
  \]
  where the bottom horizontal map is given by $\ell=s_{n-1}(\sigma)$. We will show that we can always construct the dotted arrow $\hat{t}$ and thus establishing our claim. We define an auxiliary simplicial set
  \[
    P^n=\Delta^n \coprod\limits_{\Delta^{\{n\}}}\Delta^1 \subset \Delta^{n+1}
  \]
  and inductively define simplicial sets $P^n_i$ obtained from $P^n_{i-1}$ by attaching the face missing $i$ for $0\leq i \leq n-1$ where we are using the convention $P^n_{-1}=P^n$. This yields a filtration 
  \[
    \func{P^n \to P^n_0 \to \cdots \to P^n_{n-1}=\Lambda^{n+1}_{n+1} \to \Delta^{n+1}}
  \]
  We view each $P^n_i$ as a simplicial set over $S$ by means of the composite $P^n_i \to \Delta^{n+1} \to S$ where the first map is the canonical inclusion and the second map is given by $\ell$. Suppose that we are given commutative diagram
   \[
     \begin{tikzcd}[ampersand replacement=\&]
    P^n \arrow[r,"g"] \arrow[d] \& X \arrow[d,"p"] \\
    \Delta^{n+1} \arrow[r,"\ell"] \arrow[ur,dotted]  \& S
     \end{tikzcd}
  \]
  such that $\restr{g}{\Delta^{\set{n,n+1}}}$ is $p$-Cartesian and $\restr{g}{\Delta^{\set{0,n-1,n}}}$ belongs to $C_X$. We claim that we can produce solution to this lifting problem such that such that every triangle containing the edge $n \to n+1$ is thin. The proof of the claim follows by a simple induction argument reminiscent of \autoref{lem:inductionextravaganzza}. The first case we deal with is $n=1$ which follows immediately. Then we suppose that the claim holds for $n-1$, in order to tackle the case $n$ we note that we have pushout diagrams
  \[
     \begin{tikzcd}[ampersand replacement=\&]
    P^{n-1}_{i} \arrow[r] \arrow[d] \& \Delta^n \arrow[d] \\
    P^n_i \arrow[r]   \& P^n_{i+1}{}
     \end{tikzcd}
  \]
  so the extension can be produced stepwise. Essentially the same argument shows that we can extend $t$ to the simplicial set $M^n_{n-1}$ obtained by adding to $M^n$ the faces missing $i$ for $0 \leq i \leq n-1$. Indeed for $n\geq 2$ and $i\leq n-1$ the map $M^n_{i-1} \to M^n_i$ can be expressed as a pushout of the morphism $P^{n-1}_{i-1} \to \Delta^n$. We make some important observations before continuing with the proof
  \begin{itemize}
    \item[a)] We further see that if $n\geq 3$ then $\Delta^{\set{0,n-1,n,n+1}} \to M^n_{n-1} \to X$ is a 3-simplex with all of its faces are coCartesian except possibly the face missing $n+1$. Since the last edge is $p$-Cartesian by construction it follows from \autoref{lem:fibrants5} that $\Delta^{\set{0,n-1,n}}$ is coCartesian in $X$ but since its image under $p$ is thin it follows it follows that it must be thin in $X$.
    \item[b)] We note that by construction the simplex $\Delta^{\set{n-1,n,n+1}}\to M^n_{n-1} \to X $ can be chosen to be degenerate so that the image of $\Delta^{\set{n-1,n}}$ is a degenerate edge and therefore $p$-Cartesian.
  \end{itemize}

  Finally, we note that we can produce a pushout diagram
  \[
    \begin{tikzcd}[ampersand replacement=\&]
      \Lambda^n_n \arrow[r] \arrow[d] \& \Delta^n \arrow[d] \\
      M^n_{n-1} \arrow[r] \& \Lambda^{n+1}_n
    \end{tikzcd}
  \]
and that the morphism $\Lambda^n_n \to M^n_{n-1} \to X$ maps its last edge to a $p$-Cartesian edge and the triangle $\Delta^{\set{0,n-1,n}}$ to a thin triangle in $X$ as seen in $a)$. Therefore we can produce an extension to $\Delta^n$ and consequently we can produce an extension to $\Lambda^{n+1}_n$. The extension from $\Lambda^{n+1}_n$ is guaranteed by $b)$ above.
\end{proof}

\begin{corollary}\label{cor:cart2outof3}
  Let $p:X \to S$ be a $\mathbf{O2}$-fibration and let $\sigma:\Delta^2 \to X$ be a thin 2-simplex as pictured below
  \[
    \begin{tikzcd}
                                  & b \arrow[rd, "g"] &   \\
a \arrow[ru, "f"] \arrow[rr, "h"] &                   & c
\end{tikzcd}
  \]
  Suppose that $g$ is strongly $p$-Cartesian. Then $f$ is strongly $p$-Cartesian if and only if $h$ is strongly $p$-Cartesian. 
\end{corollary}
\begin{proof}
  Combine \autoref{prop:quasi} with \cite[Lem. 2.3.8]{GHL_LaxLim} and \cite[Lem. 2.3.9]{GHL_LaxLim}.
\end{proof}

\begin{corollary}\label{cor:ghlmapping}
   Let $p:X \to S$ be a $\mathbf{O2}$-fibration. Then an edge $e:b \to c$ in $X$ is strongly $p$-Cartesian if and only if for every object $a\in X$ post-composition with $e$ induces a homotopy pullback diagram
   \[
     \begin{tikzcd}[ampersand replacement=\&]
       X(a,b) \arrow[r,"e \circ \mathblank"] \arrow[d] \& X(a,c) \arrow[d] \\
       S(p(a),p(b)) \arrow[r,"p(e) \circ \mathblank"] \& S(p(a),p(c))
     \end{tikzcd}
   \]
\end{corollary}
\begin{proof}
  Combine \autoref{prop:quasi} with \cite[Prop. 2.3.3]{GHL_LaxLim}.
\end{proof}

\begin{proposition}\label{prop:fibreo2fib}
  Let $p:X \to S$ be an $\mathbf{O2}$-fibration. Given a pair of objects $a,b \in X$ and an (strongly) $p$-Cartesian edge $e: a' \to b$ such that $p(a)=p(a')$ we have a pullback diagram in $\iCat_{\infty}$
  \[
    \begin{tikzcd}[ampersand replacement=\&]
      X_{p(a)}(a,a') \arrow[r,"e \circ \mathblank"] \arrow[d] \& X(a,b) \arrow[d] \\
      \Delta^0 \arrow[r,"p(e)"] \& S(p(a),p(b))
    \end{tikzcd}
  \]
\end{proposition}
\begin{proof}
  Let $X_{p(e)} \to \Delta^1$ denote the pullback of $X \to S$ along the map selecting the edge $p(e)$. It is easy to verify that we have a pullback diagram of simplicial sets
  \[
    \begin{tikzcd}[ampersand replacement=\&]
      X_{p(e)}(a,b) \arrow[r] \arrow[d] \& X(a,b) \arrow[d] \\
      \Delta^0 \arrow[r,"p(e)"] \& S(p(a),p(b))
    \end{tikzcd}
  \]
It follows from \autoref{prop:mappingcoCart} that the right-most vertical map is a coCartesian fibration. This in turn implies that this diagram is a pullback diagram in $\iCat_{\infty}$. Therefore to show our claim we need to show that the induced morphism
\[
  \func{X_{p(a)}(a,a') \to X_{p(e)}(a,b)}
\]
is an equivalence of $\infty$-categories. It is immediate to check that $X_{p(a)}(a,a')=X_{p(e)}(a,a')$. The claim now follows from \autoref{cor:ghlmapping}.
\end{proof}

\begin{definition}\label{def:fibrant2Cart}
  Let $p:X \to S$ be an $\mathbf{O2}$-fibration. We say that $p$ is an $\mathbf{O2C}$-fibration (outer 2-Cartesian fibration) if for every edge $e:s \to p(x)$ in $S$ there exists a $p$-Cartesian lift $\hat{e}:\Delta^1 \to X$ such that $p(\hat{e})=e$.
\end{definition}

\begin{corollary}
  Let $p: X \to S$ be an $\mathbf{O2C}$-fibration and let $p^c:X^{c} \to S$ denote the restriction to $p$ to the simplicial subset $X^{c}$ consisting only in simplices whose triangles are in $C_X$. Then $p^{c}$ is an \emph{outer Cartesian fibration} in the sense of \cite{GHL_LaxLim}. In particular, $p$ is an outer Cartesian fibration if and only if all of its triangles belong to $C_X$.
\end{corollary}

\begin{theorem}\label{thm:characterization2Fib}
  Let $p:X \to S$ be a locally fibred map equipped with a coCartesian family of triangles $C_X$. Let $E_X$ denote the collection of p-cartesian edges. Then $p$ is a $\mathbf{O2C}$-fibration if and only if the map $(X,E_X,T_X \subseteq C_X) \to (S,\sharp,T_S \subset \sharp)$ is a 2-Cartesian fibration in the sense of \autoref{def:fibrantobjects}.
\end{theorem}
\begin{proof}
  Let us suppose that $p$ is a $\mathbf{O2C}$-fibration. We need to show that $p$ has the right lifting property with respects to the maps of \autoref{def:mbsanodyne}. The only cases that are not hardcoded into the definitions are: \ref{mb:composeacrossthin},
  \ref{mb:coCartoverThin},
  \ref{mb:dualcocart2of3},
  \ref{mb:coCart2of3}, and
  \ref{mb:equivalences}. \ref{mb:composeacrossthin} follows from \autoref{cor:cart2outof3}, \ref{mb:coCartoverThin} follows from \autoref{prop:mappingcoCart}, \ref{mb:dualcocart2of3} follows from \autoref{lem:fibrants4}, \ref{mb:coCart2of3} follows from \autoref{lem:fibrants5} and finally \ref{mb:equivalences} follows from \autoref{cor:ghlmapping}. The converse is clear.
\end{proof}

\begin{proposition}
   Suppose we are given a morphism of 2-Cartesian fibrations 
  \[
    \begin{tikzcd}
     X \arrow[rr,"f"] \arrow[rd,"p"] &   & Y \arrow[ld,swap,"q"] \\
                        & S &             
    \end{tikzcd}
  \]
  Then the following statements are equivalent:
  \begin{itemize}
    \item[i)] For every $s\in S$ the map $f_s:X_s \to Y_s$ is a bicategorical equivalence.
    \item[ii)] The map $f$ is a bicategorical equivalence.
  \end{itemize}
\end{proposition}
\begin{proof}
  The implication $i) \implies ii)$ is a direct consequence of \autoref{prop:fibresdetectequivalences}. To prove the converse let us $u,v \in X$ such that $p(u)=p(v)=s$. Then it follows from \autoref{prop:fibreo2fib} that we can identify the morphism
  \[
    \func{X_s(u,v) \to Y_s(f(u),f(v))}
  \]
  with the the fibre over $\on{id}_s$ of the map
  \[
    \begin{tikzcd}
     X(u,v) \arrow[rr,"f_{uv}"] \arrow[rd] &   & Y(f(u),f(v)) \arrow[ld] \\
                        & S(s,s) &             
    \end{tikzcd}
  \]
  Since $f$ is a bicategorical equivalence it follows that $f_{uv}$ is a categorical equivalence and we can use \cite[Prop. 3.3.1.5]{HTT} to show that the map $f_s$ is fully faithful. To finish the proof we will show that $f_s$ is essentially surjective. Let $y \in Y_s$ and pick $x \in X$ together with an equivalence $\alpha:f(x) \to y$. Let us pick an inverse to $p(\alpha)$ namely $\gamma:s \to p(x)$ and a lift of $\gamma$ which we call $\beta: \hat{x} \to x$. It is easy to see that $\beta$ must be an equivalence. To finish the proof we can assemble $f(\beta)$ and $\alpha$ into a $\Lambda^2_1$ and construct a extension to $\sigma:\Delta^2 \to X$ such that the edge $\Delta^{\set{0,2}}$ belongs to $Y_s$. 
\end{proof}

\subsection{Fibrations of simplicially enriched categories}
\begin{definition}
  We say that a $\on{Set}_{\Delta}^{+}$-enriched category $\C$ is a $\iCat_{\infty}\!$-category if it is a fibrant object in the model structure of $\on{Set}_{\Delta}^{+}$-enriched categories.
\end{definition}

\begin{proposition}\label{prop:enrichefunctoriallyfibred}
  Let $f:\C \to \D$ be a fibration of $\iCat_{\infty}\!$-categories. Then the map
  \[
    \func{\Nsc(f):\Nsc(\C) \to \Nsc(\D)}
  \]
  is a functorially fibred morphism if and only if the following hold:
  \begin{itemize}
    \item[i)] For every $x,y \in \C$ the map $\C(x,y) \to \D(f(x),f(y))$ is a coCartesian fibration of $\infty$-categories.
    \item[ii)] The composition law in $\C$ preserves coCartesian edges.
  \end{itemize}
\end{proposition}
\begin{proof}
   We will show that condition $i)$ is satisfied if and only if $\Nsc(f)$ is locally fibred and that condition $ii)$ is satisfied if and only if the collection of coCartesian triangles is functorial. It is clear that the collection of coCartesian triangles in $\Nsc(\C)$ defines a marking on the mapping spaces in $\C$. It is easy to verify that if $\Nsc(f)$is locally fibred then $i)$ holds. The proof of the converse is almost the same as the proof of Propoposition 3.1.3 in \cite{GHL_LaxLim}. Using their notation we only need to observe that $\sqcap^{n-1,1}_0 \to \square^{n}$ is in our case an anodyne morphism in the coCartesian model structure.

   To finish the proof we will show that condition $ii)$ holds if and only if the collection of coCartesian triangles is functorial. Let $x,y,z \in \C$, and observe that in order to show that the map
  \[
    \func{\C(x,y) \times \C(y,z) \to \C(x,z)}
  \]
  preserves coCartesian edges it suffices to prove the particular cases where one of the two morphisms we want to compose is degenerate. Let $\ell_i:\OO^3 \to \C$ be a functor of $\iCat_{\infty}\!$-categories sending the edge corresponding to the triangle $\Delta^{\{i-1,i,i+1\}}$ to an equivalence in $\C(\ell_i(i-1),\ell_i(i+1))$. We will further assume that every edge is coCartesian except possibly the edge corresponding to the face of $\Delta^3$ missing $i$. Let us consider a pair of commutative diagrams
  \[
    \begin{tikzcd}[ampersand replacement=\&]
      013 \arrow[r]  \& 0123 \\
      03 \arrow[u,circled] \arrow[r] \& 023 \arrow[u,"\simeq"]
    \end{tikzcd} \quad \quad \quad \quad \quad  \begin{tikzcd}[ampersand replacement=\&]
      013 \arrow[r,"\simeq"]  \& 0123 \\
      03 \arrow[u] \arrow[r,circled] \& 023 \arrow[u]
    \end{tikzcd}
  \]
  that we interpreted as the image of the morphism $\OO^3(0,3) \to \C(\ell_i(0),\ell_i(3))$ for $i=1,2$. We have circled in both diagrams the coCartesian edges. Note that in both cases the full diagram will consist of coCartesian edges if and only if one of the unmarked edges is coCartesian. Let us take for example the first diagram. Then we see that if coCartesian edges are stable under precomposition along degenerate edges then the top horizontal morphism must be coCartesian. This in turn implies that the bottom horizontal morphism is coCartesian and the triangle in $\Delta^3 \to \Nsc(\C)$ missing the vertex $1$ must be coCartesian. A similar argument runs for the second diagram. This shows that $ii)$ implies that the collection of coCartesian triangles is functorial. We leave the converse as an easy exercise for the reader.
\end{proof}

\begin{definition}
  Let $f:\C \to \D$ be a map of $\iCat_{\infty}\!$-categories. An edge $e:x \to y$ is said to be $f$-Cartesian if for every $z \in \C$ the following diagram
  \[
    \begin{tikzcd}[ampersand replacement=\&]
      \C(z,x) \arrow[d] \arrow[r] \& \C(z,y) \arrow[d] \\
      \D(f(z),f(x)) \arrow[r] \& \D(f(z),f(y))
    \end{tikzcd}
  \]
  is a homotopy pullback square in $\on{Set}_{\Delta}^+$.
\end{definition}

The next theorem follows readily from \autoref{prop:enrichefunctoriallyfibred}.

\begin{theorem}\label{thm:enriched}
  Let $f:\C \to \D$ be a fibration of $\iCat_{\infty}\!$-categories. Then $\Nsc(f)$ is a 2-Cartesian fibration if and only if the following conditions hold:
  \begin{itemize}
    \item[i)] For every $x,y \in \C$ the map $\C(x,y) \to \D(f(x),f(y))$ is a coCartesian fibration of $\infty$-categories.
    \item[ii)] The composition law in $\C$ preserves coCartesian edges.
    \item[iii)] For every morphism $e:d \to f(y)$ in $\D$. There exists an $f$-Cartesian lift $\hat{e}:\hat{d} \to y$ with $f(\hat{e})=e$.
  \end{itemize}
\end{theorem}

\begin{definition}
  Let $\on{N}_2: \on{2Cat} \to \on{Set}^{+}_{\Delta}$ be the hom-wise nerve functor. We say that a functor of 2-categories $f: \CC \to \DD$ is a 2-Cartesian fibration if and only if $N_2(f)$ satisfies the conditions of \autoref{thm:enriched}.
\end{definition}

\begin{remark}
  We would like to point out that the definition above (after taking the pertinent duals) recovers the definition of 2-fibration presented in \cite{Buckley}. In particular, it follows from \autoref{thm:enriched} that our definition generalises the classical notion of a 2-fibration to the realm of $\infty$-bicategories.
\end{remark}

\end{document}